\documentclass{svproc}

 \usepackage[utf8]{inputenc}
 \usepackage[T1]{fontenc}
 \usepackage{lmodern}
 \usepackage{amsmath}
 \usepackage{amssymb}
 
 \usepackage{hyperref}

\renewcommand{\d}{\mathrm d}    
\renewcommand{\Im}{\mathrm{Im}} 
\renewcommand{\Re}{\mathrm{Re}}
\newcommand{\I}{\mathrm{I}}    
\newcommand{\T}{\mathrm{T}}    
\DeclareMathOperator{\Tr}{Tr}  
\newcommand{\eps}{\varepsilon}

\begin{document}
\mainmatter

 \title{Alignment of self-propelled rigid bodies: from particle systems to macroscopic equations}
 \titlerunning{Alignment of rigid bodies: from particle systems to macroscopic equations}
 \author{Pierre Degond\inst{1} \and Amic Frouvelle\inst{2} \and Sara Merino-Aceituno\inst{3} \and Ariane Trescases\inst{4}}
 \institute{Department of Mathematics, Imperial College London, South Kensington Campus, London, SW7 2AZ, UK, \\ \email{pdegond@imperial.ac.uk} \and CEREMADE, CNRS, Université Paris-Dauphine, Université PSL, 75016 PARIS, FRANCE, \\ \email{frouvelle@ceremade.dauphine.fr} \and School of Mathematical and Physical Sciences, University of Sussex, Falmer BN1 9RH, UK,\\
 \email{s.merino-aceituno@sussex.ac.uk}\\
Faculty of Mathematics, University of Vienna, Oskar-Morgenstern-Platz 1,
1090 Wien, Austria\\
\email{sara.merino@univie.ac.at}\\
 Department of Mathematics, Imperial College London, South Kensington Campus, London, SW7 2AZ, UK,
 \and IMT; UMR5219,
Université de Toulouse; CNRS,
F-31400 Toulouse, France,\\ \email{ariane.trescases@math.univ-toulouse.fr}}

 \maketitle
 
 \begin{abstract}
The goal of these lecture notes is to present in a unified way various models for the dynamics of aligning self-propelled rigid bodies at different scales and the links between them. The models and methods are inspired from~\cite{degond2017new,degond2018quaternions}, but, in addition, we introduce a new model and apply on it the same methods. While the new model has its own interest, our aim is also to emphasize the methods by demonstrating their adaptability and by presenting them in a unified and simplified way. Furthermore, from the various microscopic models we derive the same macroscopic model, which is a good indicator of its universality.

 \end{abstract}
\section{Introduction}

Collective behavior arises ubiquitously in nature: fish schools, flocks of birds, herds, colonies of bacteria, pedestrian dynamics, opinion formation, are just some examples.
One of the main challenges in the investigation of collective behavior is to explain its emergent properties, that is, how from the local interactions between a large number of agents, large-scale structures  and self-organization arise at a much larger scale than the agents' sizes. Kinetic theory provides a mathematical framework for the study of emergent phenomena with the rigorous derivation of equations for the large-scale dynamics (called macroscopic equations) from particle or individual-based models. The derivation of macroscopic equations establishes a rigorous link between the particle dynamics and the large-scale dynamics. Moreover, the simulation of macroscopic equations have the advantage of being, generally, computationally far more efficient than particle simulations, especially as the number of agents grows large.

Tools for the derivation of macroscopic equations were first developed in Mathematical Physics, particularly, in the framework of the Boltzmann equation for rarefied gases~\cite{cercignani2013mathematical,degond2004macroscopic,sone2012kinetic}. However, compared to the case of classical equations in Mathematical Physics, an additional difficulty arises here in the study of living systems: the lack of conservation laws. In classical physical systems, each macroscopic quantity corresponds to a conservation law (like the conservation of the total mass, momentum and energy). However, in the models that we will consider here, the number of conserved quantities is less than the number of macroscopic quantities to be determined. To overcome this difficulty we will use the methodological breakthrough presented in~\cite{degond2008continuum}: the Generalized Collision Invariant (GCI). This new concept relaxes the condition of being a conserved quantity, and has then been used in a lot of works related to alignment of self-propelled particles~\cite{bostan2017reduced,degond2013macroscopic,degond2015phase,degond2017new,degond2018quaternions,degond2013hydrodynamic,degond2017continuum,degond2011macroscopic,degond2015multi,dimarco2016selfalignment,frouvelle2012continuum}. The goal of this exposition is precisely to clearly illustrate the application of this methodology to models for collective dynamics based on alignment of the body position.

Specifically, in the models that will be considered in this exposition each agent is described by its location in the three-dimensional space and the orientation of its body, represented by a three-dimensional frame. Each agent perceives (directly or indirectly) the orientations of the bodies of the neighboring agents and tends to align with them. This type of collective motion can be found, e.g., in sperm
dynamics and animals (birds, fish), and it is a stepping stone to modeling more complex agents composed of articulated bodies (corpora~\cite{constantin2010onsager}). For more examples and applications based on body attitude coordination see~\cite{sarlette2009autonomous} and references therein.

Our  models are inspired by time-continuous versions of the Vicsek model, introduced in the 90's~\cite{vicsek1995novel}. The Vicsek model
is now a classic in the field of collective motion: self-propelled particles move at constant speed while trying to align their direction of movement with their neighbors up to some noise.  We consider time-continuous versions of the Vicsek model since they are more prone to mathematical studies, as pointed out in~\cite{degond2008continuum}. However, there is no obvious unique way of writing a time-continuous version. In~\cite{degond2008continuum} and then in~\cite{dimarco2016selfalignment}, two different continuous versions have been proposed that differ by the way agents approach the aligned state: in the first one the particles' velocities align gradually over time towards an aligned state, and in the second one the velocities make discontinuous jumps at discrete times towards an aligned state. Interestingly, both models in~\cite{degond2008continuum} and in~\cite{dimarco2016selfalignment} give rise to the same hydrodynamic/macroscopic limit (with different values for the constants in the equations). Inspired by this, here we will present two models for alignment of rigid bodies, one given by a time-continuous gradual alignment (taken from the references~\cite{degond2017new,degond2018quaternions}), and another one for alignment based on a jump process on the velocities, that we present here for the first time.

 The reason for considering here these two types of models  is the following.
The main difficulty in applying the Generalized Collision Invariant method to obtain the macroscopic equations lays, precisely, on finding the explicit form of the Generalized Collision Invariants. Indeed, in~\cite{degond2017new,degond2018quaternions} that was the main mathematical difficulty.  However, we will see that in the jump model it is straightforward to obtain the GCI but, at the same time, the computation of the macroscopic limit keeps the same structure as in the previous results~\cite{degond2017new,degond2018quaternions}. Particularly, we will obtain the same 
macroscopic equations (though with different values for the coefficients). The jump model constitutes, therefore, an excellent framework for a didactic exposition of the GCI methodology. With this, the proofs in~\cite{degond2017new,degond2018quaternions} will become more accessible to the reader.

Here, to model alignment of the orientations of the agents seen as rigid bodies (and not only the alignment of their velocities as in the original Vicsek model), we represent the body orientation of an agent as a three-dimensional frame, obtained by the rotation of a fixed frame. Therefore, we will represent the orientations of the agents as  rotations. But, as we will see in Section~\ref{sec:prelim}, in the three-dimensional space rotations can be equivalently represented by rotation matrices and unitary quaternions. Using rotation matrices, the modeling at the individual-based level is more natural and intuitive. However, in terms of numerical efficiency, quaternions require less memory usage (it only requires storing~$4$ entries rather than~$9$ entries for matrices) and are less costly to renormalize (while obtaining a rotation matrix from an approximate matrix typically requires a polar decomposition, obtaining a unit quaternion from an approximate quaternion only requires dividing by the norm). We will also see that working with quaternions can give rise to a better presentation of the macroscopic equations.

 We conclude by noting that the study of collective behavior based on the Vicsek model and its variations is a fertile field. Among many of the existing mathematical works, we highlight~\cite{degond2008continuum} where the hydrodynamic limit has been computed as well as~\cite{degond2015phase}, where the emergence of phase transitions is investigated. Many refinements have been proposed to incorporate additional mechanism, such as, to cite only a few of them, volume exclusion~\cite{degond2015multi}, presence of leaders~\cite{ferdinandy2017} or polarization of the group~\cite{cavagna2014flocking}. We refer the interested reader to~\cite{degond2017new,degond2018quaternions} and the references therein for more on this topic.

\medskip
The structure of the document is as follows.
We start by introducing some notations and recalling some useful properties on matrices, rotations and quaternions in Sec.~\ref{sec:prelim}. We present the individual-based models in Sec.~\ref{sec:micro}. From there we derive the mesoscopic models in Sec.~\ref{sec:kinetic}. Finally we compute the hydrodynamic limit in Sec.~\ref{sec:macro} and discuss the results in Sec.~\ref{sec:conclusion}.
 
 \section{Preliminaries: matrices, rotations and quaternions}\label{sec:prelim}
 We present in this section some notations and useful properties on matrices, rotations, and quaternions.
 We first introduce some notations and present some properties on matrices and rotation matrices. In a second subsection we detail the link between rotations in~$\mathbb{R}^3$ and unit quaternions.

The main drawback in using quaternion to represent a rotation is that two opposite quaternions represent the same rotation. In analogy with the theory of rodlike polymers~\cite{doi1999theory}, where two opposite unit vectors represent the same orientation, we also present in this first subsection the formalism of~$Q$-tensors, which will be helpful for the modeling.
 Most of the results will be given without detailed proofs, which can be found in Section 5 of~\cite{degond2018quaternions}.
 
\subsection{Matrices, rotation matrices and~$\mathbb{R}^3$}
 
We start by introducing a few notations. We will use the following matrix spaces:
\begin{itemize}
\item~$\mathcal{M}$ is the set of three-by-three matrices,
\item~$\mathcal{S}$ is the set of symmetric three-by-three matrices,
\item~$\mathcal{A}$ is the set of antisymmetric three-by-three matrices,
\item~$O_3(\mathbb{R})$ is the orthogonal group in dimension three,
\item~$SO_3(\mathbb{R})$ is the special orthogonal group in dimension three.
\end{itemize}
For a matrix~$A\in\mathcal{M}$, we denote by~$A^\T$ its transpose, and we write~$\Tr A$ its trace,~$\Tr A=\sum_i A_{ii}$. The matrix~$\I$ is the identity matrix in~$\mathcal{M}$. We use the following definition of the dot product on~$\mathcal{M}$: for~$A$,~$B \in \mathcal{M}$,
\[A \cdot B := \frac12 \sum_{i,j=1}^3 A_{ij}B_{ij}.\]

The choice of this dot product (note in particular the factor~$\frac12$) is motivated by the following property: for any~$u=(u_1,u_2,u_3)\in \mathbb{R}^3$, define the antisymmetric matrix~$[u]_\times$ such that
\[[u]_\times:= \begin{pmatrix} 0 &-u_3 &u_2 \\ u_3 &0 &-u_1 \\ -u_2 &u_1 &0\end{pmatrix},\]
(or equivalently such that for any~$v\in \mathbb{R}^3$, we have~$[u]_\times v = u \times v$).
Then we have for any~$u$,~$v\in \mathbb{R}^3$:
\[[u]_\times \cdot [v]_\times = u\cdot v.\]
 
The following properties will be useful in the sequel. We state them without proof but the interested reader can find them in Ref.~\cite{degond2017new}.
\begin{proposition}[Space decomposition in symmetric and antisymmetric matrices] \label{prop:spacedecompositionsymmetricandanti}
We have
$$\mathcal{S} \oplus \mathcal{A}=\mathcal{M} \mbox{ and } \mathcal{A}\perp \mathcal{S}.$$
\end{proposition}

\begin{proposition}[Tangent space to~$SO_3(\mathbb{R})$, and projection] \label{prop:tangent_space_SO}
For a matrix~$A\in SO_3(\mathbb{R})$, denote by~$T_A$ the tangent space to~$SO_3(\mathbb{R})$ at~$A$. Then 
\[M\in T_A \mbox{ if and only if there exists } P\in \mathcal{A}\mbox{ s.t } M=AP,\]
or equivalently the same statement with~$M=PA$.
Consequently, the orthogonal projection of a matrix~$M$ on~$T_A$ is given by
\begin{equation}\label{eq-PTA}P_{T_A}(M)=\frac12(M-AM^\T A),\end{equation}
and we have that~$M\in T^\perp_A$ if and only if~$M=AS$ (or equivalently~$M=SA$), for some~$S\in\mathcal{S}$.
\end{proposition}

We end up by recalling the polar decomposition of a matrix.

\begin{proposition}\label{prop:polar_dec} Let~$M \in \mathcal{M}$. There exist~$A\in O_3(\mathbb{R})$ and~$S\in \mathcal{S}$ such that
\begin{equation*}
M = AS.
\end{equation*}
Furthermore, if~$\det M \neq 0$, then~$A$ and~$S$ are unique. In this case, we write
\begin{equation*}
PD(M) := A.
\end{equation*}
 \end{proposition}
 
\subsection{Quaternions, rotations and~$Q$-tensors}
 Besides rotation matrices, another common representation of rotations in~$\mathbb{R}^3$ is done through the unit quaternions, which will be denoted by~$\mathbb{H}_1$. Recall that any quaternion~$q$ can be written as~$q=a+b\mathrm{i}+c\mathrm{j}+d\mathrm{k}$ with~$a,b,c,d\in\mathbb{R}$. Quaternions form a four dimensional (non commutative) division algebra, by the rules~$\mathrm{i}^2=\mathrm{j}^2=\mathrm{k}^2=\mathrm{ijk}=-1$. The real part~$\Re(q)$ of the quaternion~$q$ is~$a$ and its imaginary part, denoted~$\Im(q)$  is~$b\mathrm{i}+c\mathrm{j}+d\mathrm{k}$. The three-dimensional space of purely imaginary quaternions is then identified with~$\mathbb{R}^3$, therefore whenever in the paper we have a vector in~$\mathbb{R}^3$ which is used as a quaternion, it should be understood that it is a purely imaginary quaternion thanks to this identification. For instance, the vector~$e_1\in\mathbb{R}^3$ (resp.~$e_2$,~$e_3$) is identified with the quaternion~$\mathrm{i}$ (resp.~$\mathrm{j}$,~$\mathrm{k}$). The conjugate of the quaternion~$q$ is given by~$q^*=\Re(q)-\Im(q)$, therefore we get~$qq^*=|q|^2=a^2+b^2+c^2+d^2\geqslant0$.

We now explain how the group~$\mathbb{H}_1$ (the unit quaternions~$q$, such that~$|q|=1$)
provides a representation of rotations.
Any unit quaternion~$q\in\mathbb{H}_1$ can be written in a polar form as~$q = \cos(\theta/2) + \sin(\theta/2) n,$ where~$\theta\in[0,2\pi)$ and~$n\in \mathbb{S}^2$ (a purely imaginary quaternion with the previous identification). With this notation, the unit quaternion~$q$ represents the rotation of angle~$\theta$ around the axis given by the direction~$n$, anti-clockwise. More specifically, for any vector~$u\in \mathbb{R}^3$, the vector~$quq^*$ (which is indeed a pure imaginary quaternion whenever~$q\in\mathbb{H}_1$ and~$u$ is a pure imaginary quaternion, so it can be seen as a vector in~$\mathbb{R}^3$) is the rotation of~$u$ of angle~$\theta$ around the axis given by the direction~$n$ (note that~$\theta$ and~$n$ are uniquely defined except when~$q=\pm1$: in this case the associated rotation is the identity, and any direction~$n\in \mathbb{S}^2$ is suitable). 

The underlying map from the group of unit quaternions to the group of rotation matrices is then given by
 \begin{equation}\label{def:Phi}
 \Phi: \begin{array}{ccl}\mathbb{H}_1 &\to& SO_3(\mathbb{R})\\
 q &\mapsto& \Phi(q):\begin{array}{ccl}\mathbb{R}^3&\to&\mathbb{R}^3\\u&\mapsto &quq^*.\end{array}\end{array}
 \end{equation}
 It is then straightforward to get that~$\Phi$ is a morphism of groups: for any~$q$ and~$\widetilde{q}$ in~$\mathbb{H}_1$, we have~$\Phi(q\widetilde{q})=\Phi(q)\Phi(\widetilde{q})$ and~$\Phi(q^*)=\Phi(q)^\T$.
 
 An important remark is that two opposite unit quaternions represent the same rotation:
 \begin{equation}\label{eq:q_and_opposite}
\forall q \in \mathbb{H}_1, \quad \Phi(q)=\Phi(-q).
\end{equation}
 More precisely, the kernel of~$\Phi$ is given by~$\{\pm 1\}$, so that~$\Phi$ induces an isomorphism between~$\mathbb{H}_1/\{\pm 1\}$ and~$SO_3(\mathbb{R})$.

We finally briefly introduce the notion of~$Q$-tensors. Indeed, since a unitary quaternion and its opposite correspond to the same rotation matrix, we can see an analogy with the theory of suspensions of rodlike polymers~\cite{doi1999theory}. Those polymers are also modeled using unit vectors (in this case, in~$\mathbb{R}^3$), and two opposite vectors are describing the same orientation. Their alignment is called nematic. One relevant object in this theory is the so-called~$Q$-tensor associated with the unit quaternion~$q$, given by the matrix~$Q=q\otimes q-\frac1{4}\I_4$, where~$q$ is seen as a unit vector in~$\mathbb{R}^4$, and~$\I_4$ is the identity matrix of size four. This object is a symmetric and trace free four by four matrix, which is invariant under the transformation~$q\mapsto-q$. We denote by~$\mathcal{S}_4^0$ the space of symmetric trace free~$4\times4$ matrices (a vector space of dimension~$9$), and endow it with the dot product known as “contraction of tensors”; more precisely if~$Q,\widetilde{Q}$ are in~$\mathcal{S}_4^0$, their contraction~$Q:\widetilde{Q}=\sum_{i,j}Q_{ij}\widetilde{Q}_{ij}$ is the trace of~$Q\widetilde{Q}^\T$. We then get a map 
 \begin{equation*}\label{def:Psi}
 \Psi: \begin{array}{ccl}\mathbb{H}_1 &\to& \mathcal{S}_4^0\\
 q &\mapsto&q\otimes q-\frac1{4}\I_4,\end{array}
 \end{equation*}
whose image can also be identified with~$\mathbb{H}_1/\{\pm1\}$. Indeed, the preimage of~$\Psi(q)$ is always equal to~$\{q,-q\}$. We therefore have two ways to see~$\mathbb{H}_1/\{\pm1\}$ as a submanifold of a nine-dimensional vector space : either as the image of~$\Phi$ (in~$\mathcal{M}$), which is exactly~$SO_3(\mathbb{R})$, or as the image of~$\Psi$ (in~$\mathcal{S}_4^0$). It appears that the dot products on these spaces behave remarkably well, regarding the maps~$\Phi$ and~$\Psi$, as stated in the following proposition, from which we can also see that the images are submanifolds of the spheres of radii~$\sqrt{\frac32}$ (in~$\mathcal{M}$) and~$\frac{\sqrt3}2$ (in~$\mathcal{S}_4^0$).
 
 \begin{proposition}\label{prop-dot-products} For any unit quaternions~$q$ and~$\widetilde{q}$, we have
  \[\frac12\Phi(q)\cdot\Phi(\widetilde{q})=(q\cdot\widetilde{q})^2-\frac14=\Psi(q):\Psi(\widetilde{q}).\]
 \end{proposition}
\begin{proof}
For the second equality, recall that for any quaternions~$q$ and~$\widetilde{q}$ we have by definition~$(q\otimes\widetilde{q})_{ii}=q_i\widetilde{q}_i$, therefore~$\Tr(q\otimes\widetilde{q})=q\cdot\widetilde{q}$ (this justifies the fact that~$\Tr(q\otimes q-\frac14\I_4)=0$ when~$q$ is a unit quaternion). Using the fact that~$(q\otimes q)(\widetilde{q}\otimes\widetilde{q})=(q\cdot\widetilde{q})\,q\otimes\widetilde{q}$, we get, when~$q$ and~$\widetilde{q}$ are unit quaternions:
\[\begin{split}\Psi(q):\Psi(\widetilde{q})&=\Tr((q\otimes q-\tfrac14\I_4)(\widetilde{q}\otimes\widetilde{q}-\tfrac14\I_4))\\&=\Tr((q\otimes q)(\widetilde{q}\otimes\widetilde{q}-\tfrac14\I_4))=(q\cdot\widetilde{q})^2-\frac14.\end{split}\]

For the first equality, we first prove that~$\Tr(\Phi(q))=4\Re(q)^2-|q|^2$ for any quaternion~$q$. Indeed we first have that~$\Tr(\Phi(q))=\sum_{i=1}^3e_i\cdot\Phi(q)e_i$. Writing~$q$ as~$a+b\mathrm{i}+c\mathrm{j}+d\mathrm{k}$ and using the identifications between~$\mathbb{R}^3$ and the purely imaginary quaternions, we get for instance that~$e_1\cdot\Phi(q)e_1=\Re(\mathrm{i}(q\mathrm{i}q^*)^*)$, which after computations is~$a^2+b^2-c^2-d^2$. At the end, with similar computations for~$e_2$ and~$e_3$, we get~$\Tr(\Phi(q))=3a^2-b^2-c^2-d^2=4\Re(q)^2-|q|^2$. Therefore, if~$q$ and~$\widetilde{q}$ are unit quaternions, we get \[\begin{split}2\Phi(q)\cdot\Phi(\widetilde{q})&=\Tr(\Phi(q)\Phi(\widetilde{q})^T)\\
&=\Tr(\Phi(q\widetilde{q}^*))=4\Re(q\widetilde{q}^*)^2-|q\widetilde{q}^*|^2=4(q\cdot\widetilde{q})^2-1.\qed\end{split}\]
\end{proof}
We will also make use of the two following properties regarding the differentiability of the map~$\Phi$ and the volume forms in~$SO_3(\mathbb{R})$ and~$\mathbb{H}_1$.
  
 \begin{proposition}\label{prop:diff_Phi}
The map~$\Phi$ is continuously differentiable on~$\mathbb{H}_1$. Denoting its differential at~$q \in \mathbb{H}_1$ by~$\textnormal{D}_q \Phi: q^\perp \longrightarrow T_{\Phi(q)}$ , we have that for any~$p\in q^\perp$,
\begin{equation*}
\label{eq:DPhi_is_uPhi0}
 \textnormal{D}_q\Phi (p) = 2 \left[ pq^\ast \right]_\times \Phi(q).
 \end{equation*}
Here, we wrote~$T_{A}$ the tangent space of~$SO_3(\mathbb{R})$ at~$A=\Phi(q)$, and~$q^\perp$ the orthogonal of~$q$.
 \end{proposition}
 
 \begin{proposition}\label{prop:meas_Phi}
Consider a function~$g:SO_3(\mathbb{R})\to \mathbb{R}$, then
\begin{equation*}
 \label{eq:equivalence integrals}
\int_{SO_3(\mathbb{R})}g(A)\, \mathrm{d} A= \int_{\mathbb{H}_1}g(\Phi(q))\,\mathrm{d} q,
\end{equation*}
where~$\d q$ and~$\d A$ are the normalized Lebesgue measures on the hypersphere~$\mathbb{H}_1$ and on~$SO_3(\mathbb{R})$, respectively. Furthermore, if~$B\in SO_3(\mathbb{R})$, then
\begin{equation*}
 \label{eq:left_right_invariance}
\int_{SO_3(\mathbb{R})}g(A)\, \mathrm{d} A=\int_{SO_3(\mathbb{R})}g(AB)\, \mathrm{d} A=\int_{SO_3(\mathbb{R})}g(BA)\, \mathrm{d} A.
\end{equation*}

 \end{proposition}
 \section{Individual Based Modeling: alignment of self-propelled rigid bodies}\label{sec:micro}
Our goal is to model a large number~$N$ of particles described, for~$n=1,\dots,N$, by their positions~$X_n\in\mathbb{R}^3$ and their orientations as rigid bodies. The most natural way to describe such an orientation is to give three orthogonal unit vectors~$u_n$,~$v_n$, and~$w_n$. For instance, one way to describe the full orientation of a bird, would be to set the first vector~$u_n$ as the direction of its movement (from the center to the beak), the second vector~$v_n$ as the direction of its left wing (from the center to the wing), and the last one~$w_n$ as the direction of the back, so that~$u_n$,~$v_n$ and~$w_n$ form a direct orthogonal frame. Therefore the matrix~$A_n$ whose three columns are exactly~$u_n$,~$v_n$ and~$w_n$ is a special orthogonal matrix. This rotation matrix~$A_n$ represents the rotation that has to be done between a reference particle the orthogonal vectors of which are exactly the canonical basis of~$\mathbb{R}^3$ (denoted by~$e_1=(1,0,0)$,~$e_2=(0,1,0)$, and~$e_3=(0,0,1)$), and the particle number~$n$.
A particle can then be described by a pair~$(X_n,A_n)\in\mathbb{R}^3\times SO_3(\mathbb{R})$.

 In the spirit of the Vicsek model~\cite{vicsek1995novel}, we want to include in the modeling the three following rules :
 \begin{itemize}
   \item particles move at constant speed,
   \item particles try to align with their neighbors,
   \item this alignment is subject to some noise.
 \end{itemize} 
   
 Up to changing the time units, we will consider that all particles move at speed one. The first rule requires a direction of movement for each particle. Therefore, in the following, we will suppose that the first vector~$u_n=A_ne_1$ of the matrix~$A_n$ represents the velocity of the particle number~$n$. Then, the evolution of the position~$X_n$ will simply be given by
 \begin{equation}\label{eq-dXndt}
   \frac{\d X_n}{\d t}=A_ne_1.
 \end{equation}
 
 In the quaternion framework, if the quaternion~$q_n$ represents the orientation of particle number~$n$ (meaning that~$\Phi(q_n)=A_n$) then the equation corresponding to~\eqref{eq-dXndt} reads:
 \begin{equation}\label{eq-dXndt-q}
   \frac{\d X_n}{\d t}=q_ne_1q_n^*.
 \end{equation}

We now want to describe the evolution of~$q_n$ (or of the rotation matrix~$A_n$), taking into account the two remaining rules.
 
 \subsection{Defining the target for the alignment mechanism}\label{subsec:target}
 To implement the second rule in the modeling, in the spirit of the Vicsek model, we need to provide for each particle a way to compute the ``average orientation'' of the neighbors. In the Vicsek model, the idea was to take the sum of all the velocities of the neighbors and to normalize it in order to have a unit target velocity.

 In our framework of rotation matrices, to apply the same procedure, if we want the target orientation~$\bar{A}_n$ (viewed from the particle number~$n$) to be a rotation matrix, we need a procedure of normalization which from any matrix gives a matrix of~$SO_3(\mathbb{R})$. Indeed, the sum of all rotations matrix~$A_m$ of the neighbors need not be a rotation matrix (nor a multiple of such a matrix). The choice that had been done in~\cite{degond2017new} was to take the polar decomposition : we denote
 \begin{align}\label{eq-def-Jbarn}
   \bar{J}_n&=\frac1N\sum_{m=1}^NK(X_m-X_n)A_m\\
   \bar{A}_n&=\mathrm{PD}(\bar{J}_n),\label{eq-baddef-Abarn}
 \end{align}
 where~$\mathrm{PD}(J)$ (when~$\det(J)\neq 0$) denotes the orthogonal matrix in the polar decomposition (see Prop.~\ref{prop:polar_dec}) of a matrix~$J$, and~$K$ is an observation kernel, which weights the orientations of neighbors. A simple example is~$K(x)=1$ if~$0\leqslant|x|<R$ and~$K(x)=0$ if~$|x|\geqslant R$. In that case all the neighbors located in the ball of radius~$R$ and center~$X_n$ have the same influence on the computation of the average matrix~$\bar{A}_n$, and all individuals located outside this ball have no influence on this computation.

 A first difficulty arises here when the polar decomposition is not a rotation matrix, that is to say~$\det(\bar{J}_n)\leqslant0$. Indeed, due to random effects, we can expect that the polar decomposition is almost surely defined (that is to say~$\det(\bar{J}_n)\neq0$, which happens on a negligible set), but we cannot expect that~$\det(\bar{J}_n)>0$ almost surely.

 In the framework of unitary quaternions, things are slightly more complicated. Indeed, since a unitary quaternion and its opposite correspond to the same rotation matrix (see eq.~\eqref{eq:q_and_opposite}), the expression used to compute an ``average orientation'' needs to be invariant by the change of sign of any of the quaternions appearing in the formula. Using the~$Q$-tensors as in the theory of suspensions of rodlike polymers~\cite{doi1999theory} is a good option. We are then led to averaging objects of the form~$q\otimes q-\frac1{4}\I_4$ which are invariant under the transformation~$q\mapsto-q$. The average~$Q$-tensor of the neighbors would then take the form

 \begin{equation}\label{eq-def-Qbarn}
   \bar{Q}_n=\frac1N\sum_{m=1}^NK(X_m-X_n)(q_m\otimes q_m-\tfrac1{4}\I_4).
 \end{equation}

 To define now an ``average'' quaternion from this~$Q$-tensor~$\bar{Q}_n$, we need a procedure which provides a unit vector. We expect that if all quaternions~$q_n$ are all equal to a given~$q$ (or to~$-q$), the procedures returns~$q$ or~$-q$. Therefore, from the form~$\alpha(q\otimes q-\frac1{4}\I_4)$, with~$\alpha>0$, it should return~$q$ or~$-q$. These two vectors are precisely the unit eigenvectors associated to the maximal eigenvalue (which is equal to~$\frac34$, the other eigenvectors, orthogonal to~$q$, being associated to the eigenvalue~$\frac{-1}4$). Therefore, in~\cite{degond2018quaternions}, we defined
 \begin{equation}\label{eq-def-qbarn}
   \bar{q}_n = \text{one of the unit eigenvectors of }\bar{Q}_n\text{ of maximal eigenvalue}.
 \end{equation}
 Note that the direction of~$\bar{q}_n$ is uniquely defined when the maximal eigenvalue is simple. Since symmetric matrices with multiple maximal eigenvalues are negligible, we can expect this definition to be well-posed almost surely.

 The first unexpected link found in~\cite{degond2018quaternions} between this framework of quaternions and the previous framework of average matrices, is that these two averaging procedures are actually equivalent (when the polar decomposition in formula~\eqref{eq-baddef-Abarn} actually returns a rotation matrix). 
 This is due to the following observations~\cite{degond2018quaternions}:

 \begin{proposition}~ \label{prop:equiv_averaging}
   \begin{enumerate}
   \item[(i)] If~$M\in\mathcal{M}$ is such that~$\det(M)>0$, then the polar decomposition of~$M$ is a rotation matrix, and it is the unique maximizer of the function~$A\mapsto A\cdot M$ among all matrices~$A$ in~$SO_3(\mathbb{R})$.
     \item[(ii)] A unit eigenvector corresponding to the maximal eigenvalue of a symmetric matrix~$Q$ maximizes the function~$q\mapsto q\cdot Qq$ among all unit vectors~$q$ of~$\mathbb{H}_1$.
     \item[(iii)] If for all~$n$ we have~$\Phi(q_n)=A_n$, and~$\det(\bar{J}_n)>0$, then~$\Phi(\bar{q}_n)=\bar{A}_n$, where~$\bar{J}_n$,~$\bar{A}_n$ and~$\bar{q}_n$ are given by~\eqref{eq-def-Jbarn}-\eqref{eq-baddef-Abarn}, and~\eqref{eq-def-Qbarn}-\eqref{eq-def-qbarn}.
     \end{enumerate}
   \end{proposition}
   
   We therefore have now a good procedure to compute~$\bar{A}_n$ thanks to the following maximization problem (instead of polar decomposition):
   \begin{equation}\label{eq-gooddef-Abarn}
   \bar{A}_n=\textnormal{argmax}\left\{A\in SO_3(\mathbb{R}) \mapsto A\cdot M_n\right\},
 \end{equation}
 where~$\bar{J}_n$ is defined in Eq.~\eqref{eq-def-Jbarn}, and from now on we use this definition of~$\bar{A}_n$, which ensures that it corresponds to the definition~\eqref{eq-def-qbarn} of~$\bar{q}_n$ in the world of quaternions.
 
 Since the next part of the modeling will include some random effects, we can expect that the configurations for which the average is not well-defined will be of negligible probability.
 
   We now need to have evolution equations for the orientations (either the rotation matrices~$A_n$ or the unit quaternions~$q_n$).
   In the spirit of the time discrete Vicsek model, it would correspond to saying~$A_n(t+\mathrm{\Delta}t)=\bar{A}_n(t)+$ ``noise'', or~$q_n(t+\mathrm{\Delta}t)=\bar{q}_n(t)+$ ``noise''. However, as was pointed out in~\cite{degond2008continuum}, in this procedure~$\mathrm{\Delta}t$ is actually a parameter of the model, and not a time discretization of an underlying process : indeed, this parameter controls the frequency at which particles change their orientation, and changing the value of this frequency leads to drastic changes in the behavior of the model. Regarding the mathematical study of this type of time-discrete models, it is far from being clear how to go beyond observations of numerical simulations. However, it is possible to build models in the same spirit as the Vicsek model which will be much more prone to mathematical study, in particular if we want to derive a kinetic description (when the number of particles is large) and macroscopic limit (when the scale of observation is large). In the next two subsections we present two ways of building such models. The first one corresponds to a time-continuous alignment mechanism as was proposed in~\cite{degond2008continuum}, in which the orientation of one particle continuously tries to align with its target orientation, up to some noise. This leads to the models presented in~\cite{degond2017new} (in the framework of rotation matrices) and in~\cite{degond2018quaternions} (in the framework of unit quaternions).
   The second one, as in the Vicsek model, corresponds to a process in which orientations undergo jumps as time evolves, but where the jumps are not synchronous: instead of taking place every time step for all particles, they all have independent times at which they change from their orientations to their target orientations, up to some noise. This leads to a new model, which is different at the particle and kinetic levels, but for which the derivation of the macroscopic model gives the same system of evolution equations (up to the values of the constant parameters of the model). This procedure was studied in~\cite{dimarco2016selfalignment} for the alignment mechanism of the Vicsek model, and they also found that the macroscopic model corresponds to the Self-Organized Hydrodynamic system of~\cite{degond2008continuum} (derived from the time continuous alignment process).
   
   \subsection{Gradual alignment model}

    We first consider a time-continuous alignment mechanism. We have to take into account the last two rules (particles try to align with their neighbors, and this alignment is subject to some noise). For the sake of simplicity, we first present the alignment dynamics without noise. We will add noise at the end of this subsection.
    
    The alignment is modeled by a gradual alignment of an agent's body orientation towards its local average defined in the previous subsection. We express the evolution towards the average as the gradient of a polar distance between the agent and the average. It takes the form, in the world of rotation matrices
     \begin{equation*}\label{eq:ed_A}
    \frac{\mathrm{d} A_n}{\mathrm{d}t} = \nabla_{A_n} \left[A_n \cdot \bar A_n\right],
    \end{equation*}
    and in the world of unit quaternions
    \begin{equation*}\label{eq:ed_q}
    \frac{\mathrm{d} q_n}{\mathrm{d}t} = \nabla_{q_n} \left[\frac 12 (q_n\cdot \bar{q}_n)^2\right],
    \end{equation*}
    where the strength of alignment (or equivalently the relaxation frequency) has been taken to be one (which can be done without loss of generality by changing time units), and~$\nabla_{A_n}$ and~$\nabla_{q_n}$ represent the gradients on~$SO_3(\mathbb{R})$ and~$\mathbb{H}_1$ respectively. For the quaternions, we took the square of the norm to account for the fact that only the directions of the vectors~$q_n$ and~$\bar q_n$, and not their sign, should influence the alignment dynamics (this is called nematic alignment, and it is analogous to the case of rodlike polymers, as described in Subsec.~\ref{subsec:target}).
    
    The alignment forces can be rewritten respectively as
        \begin{equation*}
     \nabla_{A_n} \left[A_n \cdot \bar A_n\right] =  P_{T_{A_n}} \bar A_n,
    \end{equation*}
    for the matrices, where~$P_{T_{A_n}}$ is the orthogonal projection on the tangent space of~$SO_3(\mathbb{R})$ at~$A_n$, given by Eq.~\eqref{eq-PTA} (see Prop.~\ref{prop:tangent_space_SO}), and
    \begin{equation*}
     \nabla_{q_n} \left[\frac 12 (q_n\cdot \bar{q}_n)^2\right] =  P_{q_n^\perp} \left[ \left(\bar{q}_n \otimes \bar{q}_n -\tfrac14 \I_4 \right)q_n \right].
    \end{equation*}
    for the quaternions, where~$P_{q_n^\perp}=\I_4-q_n\otimes q_n$ is the projection on the orthogonal of~$q_n$.
    
    The second link found in~\cite{degond2018quaternions} between the frameworks of quaternions and rotation matrices, is that these alignment mechanisms are also equivalent.     
     \begin{proposition}
     \label{prop-equiv-nonoise}
    Consider the system, for all~$n=1..N$,
    \begin{align}\label{Cauchy_A_1}
    &\frac{\d A_n}{\d t} =  \nabla_{A_n} \left[A_n \cdot \bar A_n\right],\\
    &A(t=0)=A_n^0 \in SO_3(\mathbb{R}),
    \label{Cauchy_A_2}\end{align}
    with~$\bar A_n$ defined in~\eqref{eq-gooddef-Abarn}, and the system, for all~$n=1..N$,
    \begin{align}\label{Cauchy_q_1}
    &\frac{\d q_n}{\d t}=  \nabla_{q_n} \left[\frac 12 (q_n\cdot \bar{q}_n)^2\right],\\
    &q(t=0)=q_n^0\in \mathbb{H}_1,
    \label{Cauchy_q_2}\end{align}
    with~$\bar q_n$ defined in~\eqref{eq-def-qbarn}.
    If~$A_n^0=\Phi(q_n^0)$ for~$n=1..N$, then, for any solution~$(q_n)_n$ of the Cauchy problem~\eqref{Cauchy_q_1}--\eqref{Cauchy_q_2}, the~$N$-tuple~$(A_n)_n:=(\Phi(q_n))_n$ is a solution of the Cauchy problem~\eqref{Cauchy_A_1}--\eqref{Cauchy_A_2}.
    \end{proposition}
    
    The proof of this proposition relies on two main properties: the equivalence of the averaging procedures of Prop.~\ref{prop:equiv_averaging} on one hand, and, on the other hand, the computation of the differential of~$\Phi$ in Prop.~\ref{prop:diff_Phi}, which allows us to write a link between the gradient operators on~$SO_3(\mathbb{R})$ and on~$\mathbb{H}_1$.
    
    We finally describe the complete model by adding the third rule (the fact
    that the alignment is subject to some noise). The natural way to introduce it is to transform the ordinary differential equations~\eqref{eq-dXndt}-\eqref{Cauchy_A_1} and~\eqref{eq-dXndt-q}-\eqref{Cauchy_q_1}, into stochastic differential equations, which take the form of the two following systems:
    \begin{equation}\label{IBM_cont_A}
    \begin{cases}
    \d X_n=A_ne_1\, \d t,\\
    \d A_n = P_{T_{A_n}} \circ \left[  \bar A_n \d t + 2\sqrt{D}\, \d B_t^{9,n} \right],
    \end{cases}
    \end{equation}
    and
    \begin{equation}\label{IBM_cont_q}
    \begin{cases}
    \d X_n=q_ne_1q_n^* \,\d t,\\
    \d q_n = P_{q_n^\perp} \circ \left[  \left(\bar{q}_n \otimes \bar{q}_n -\tfrac14 \I_4 \right)q_n \d t + \sqrt{D/2} \,\d B_t^{4,n} \right],
    \end{cases}
    \end{equation}
    
    where~$(B_t^{9,n})_n$ are matrices of~$\mathcal{M}$ with coefficients given by standard independent Brownian motions, and~$(B_t^{4,n})_n$ are independent standard Brownian motions on~$\mathbb{R}^4$,~$D>0$ representing the noise intensity. The stochastic differential equations have to be understood in the Stratonovich sense, which is well adapted to write stochastic processes on manifolds~\cite{hsu2002stochastic}.
    
    \begin{theorem}[Equivalence in law~\cite{degond2018quaternions}]\label{th:equiv_law}
    The processes~\eqref{IBM_cont_A} and~\eqref{IBM_cont_q} are equivalent in law.
    \end{theorem}
    
    This theorem relies on the properties of the map~$\Phi$ defined in~\eqref{def:Phi}, in the same way as they are used to prove the equivalence of the alignment dynamics alone in Proposition~\ref{prop-equiv-nonoise}. However in that case the trajectories were exactly the same due to the uniqueness of the solution of the Cauchy problem. Here, since the driving Brownian motions do not belong to the same space (one is on a nine-dimensional space, the other one in a four-dimensional one) we cannot easily give a sense to some pathwise equivalence. However, the projection of these driving Brownian motions on the tangent space of the manifold we consider produce process which are actually three-dimensional, in the sense that their trajectories are contained in a three-dimensional manifold. This is why the equivalence is at the level of the law of the trajectories. Working on the partial differential equations satisfied by the densities of the laws of the processes, and relying on the equivalence of measures in Prop.~\ref{prop:meas_Phi}, we can make further use of the differential properties of~$\Phi$ to write a link between the divergence and Laplacian operators on~$SO_3(\mathbb{R})$ and on~$\mathbb{H}_1$. We obtain that these partial differential equation are equivalent, which give the equivalence in law. More precise details on the law of such a stochastic differential equation is given in Subsection~\ref{subsec-single-gradual} for the case of a single individual evolving in a given orientation field.
    
   \subsection{Alignment model with orientation jumps}\label{subsec:jumps}

   In this section we describe an alternative alignment mechanism where the orientations of the particles make jumps at random times. For this, we attach a Poisson point process with parameter~$1$ to each particle~$n$ (for~$n=1,\hdots, N$), which corresponds to the times at which this particle updates its orientation. The increasing sequence of positive times will be denoted by~$(t_{n,m})_{m\geqslant1}$, and can be constructed by independent increments between two consecutive times, given by exponential variables of parameter~$1$. This means that the unit of time has been chosen in order that it corresponds to the average of the time between two jumps of a given particle. 

   Next we need to define how the orientation ($A_n$ or~$q_n$) of a particle changes when there is a jump. Recall the definition of the averages~$\bar A_n$ and~$\bar q_n$ in~\eqref{eq-gooddef-Abarn} and~\eqref{eq-def-qbarn} respectively. We want the new orientation to be drawn according to a probability ``centered'' around~$\bar{A}_n$ (resp.~$\pm \bar{q}_n$) and radially symmetric, that is, it should have a density of the form~$A\mapsto M_{\bar{A}_n}(A)$ (resp.~$q\mapsto\widetilde{M}_{\bar{q}_n}(q)$), which only depends on the distance between~$A$ and~$\bar{A}_n$ (resp. the distance between~$\pm q$ and~$\pm \bar{q}_n$).
   In the matrix world, the square of the norm of an orthogonal matrix is~$\frac12\Tr(A^\T A)=\frac32$, therefore we have~$\|A-\bar{A}_n\|^2=3-2A\cdot\bar{A}_n$, we are thus looking at a probability density only depending on~$A\cdot\bar{A}_n$. Thanks to Prop.~\eqref{prop-dot-products}, in the world of quaternions, it corresponds to a function only depending on~$(q\cdot\bar{q}_n)^2$.
   
   To fix the ideas, and to see analogies with the gradual alignment model, we will take for~$M_{\bar{A}_n}$ the von-Mises distribution centered around~$\bar{A}_n$ and with concentration parameter~$\frac1D$. We will indeed see that in this case, the results of the computations for the macroscopic limits that were done in~\cite{degond2017new,degond2018quaternions} can directly be reused. Of course, the method that we present here still applies for a generic smooth function of~$A\cdot\bar{A}_n$.
   
   The von-Mises distribution centered in~$\Lambda\in SO_3(\mathbb{R})$ and with concentration parameter~$\frac1D$ is defined, for~$A$ in~$SO_3(\mathbb{R})$, by
\begin{equation}\label{eq:von_Mises_distribution}M_{\Lambda}(A) = \frac{1}{Z}\exp\left( \frac{1}{D}A\cdot \Lambda \right), \quad \int_{SO_3(\mathbb{R})} M_{\Lambda}(A)\ dA =1,
\end{equation}
where~$Z=Z_D<\infty$ is a normalizing constant such that this function is a probability density on~$SO_3(\mathbb{R})$.

Analogously, we define the von-Mises distribution on~$\mathbb{H}_1$ as
\begin{equation}\label{eq:von_Mises_q}
M_{\bar q}(q) = \frac{1}{Z'}\exp \left( \frac{2}{D}\left( (\bar q \cdot q)^2 - \frac{1}{4}\right)\right), \qquad\int_{\mathbb{H}_1}M_{\bar q}(q) \ dq =1,
\end{equation}
where~$Z'=Z'_D$ is a normalizing constant. Thanks to Prop.~\ref{prop-dot-products} and Prop.~\ref{prop:meas_Phi}, if~$q$ is a random variable on~$\mathbb{H}_1$ distributed according to~$M_{\bar{q}}$, then~$A=\Phi(q)$ is a random variable on~$SO_3(\mathbb{R})$ distributed according to~$M_{\Lambda}$, where~$\Lambda=\Phi(\bar{q})$.

   A useful property of~$SO_3(\mathbb{R})$ (or~$\mathbb{H}_1$) is that the dot product is invariant by multiplication : we have that~$M_{\bar{A}_n}(A)=M_{\I_3}(\bar{A}_n^\T A)$, since~$\I_3\cdot(\bar{A}_n^\T A)=\bar{A}_n\cdot A$. Furthermore, the measure on~$SO_3(\mathbb{R})$ is also left-invariant. We therefore only need to be able to draw random variable according to~$M_{\I_3}$, thanks to the following proposition.
   \begin{proposition}\label{prop-group-distribution}
     If~$B\in SO_3(\mathbb{R})$ is a random variable distributed according to the density~$M_{\I_3}$, then~$\bar{A}_nB$ is a random variable distributed according to the density~$M_{\bar{A}_n}$.
     
     Analogously, if~$r\in\mathbb{H}_1$ is a random variable distributed according to the density~$M_{1}$, then~$\bar{q}_nr$ is a random variable distributed according to the density~$M_{\bar{q}_n}$. 
   \end{proposition}
   \begin{proof} If~$U$ is a measurable set of~$SO_3(\mathbb{R})$, then, by left invariance of the measure
     \[\begin{split}\mathbb{P}(\bar{A}_nB\in U)&=\mathbb{P}(B\in\bar{A}_n^\T U)\\
     &=\int_{\bar{A}_n^\T U}M_{\I_3}(A)\d A=\int_{U}M_{\I_3}(\bar{A}_n^\T A)\d A=\int_{U}M_{\bar{A}_n}(A)\d A.\end{split}\]
     Notice that this proof does not rely on the particular expression of the von-Mises distribution, and still applies if~$M_{\bar{A}_n}(A)$ is a generic function of~$\bar{A}_n\cdot A$. The proof is analogous for the quaternion version.\qed
   \end{proof}

   We are now ready to construct the stochastic process corresponding to the evolution of positions and orientations of the particles.
   \begin{definition}\label{def-jump-process}
     We are given:
     \begin{itemize}
       \item a probability density~$M_{\I_3}$ on~$SO_3(\mathbb{R})$, with the property that~$M_{\I_3}(A)$ only depends on~$\I_3\cdot A=\frac12Tr(A)$ (we will take the von-Mises distribution defined in~\eqref{eq:von_Mises_distribution} in the following of the paper), 
       \item some independent random variables~$S_{n,m}>0$ and~$\eta_{n,m}\in SO_3(\mathbb{R})$, such that for~$1\leqslant n\leqslant N$ and~$m\in\mathbb{N}$,~$S_{n,m}$ is distributed according to an exponential law of parameter~$1$ and~$\eta_{n,m}$ is distributed according to~$M_{\I_3}$,
       \item some initial positions~$X_{n,0}\in \mathbb{R}^3$ and initial body orientations~$A_{n,0}\in SO_3(\mathbb{R})$ for~$1\leqslant n\leqslant N$.
       \end{itemize}
       The variables~$(S_{n,m})_{m\in\mathbb{N}}$ represent the intervals of time between consecutive jumps for particle number~$n$. Therefore we define~$t_{n,m}=\sum_{0\leqslant\ell<m}S_{n,\ell}$, which corresponds to the time at which particle number~$n$ changes its orientation for the~$m$-th time. The positions and orientations are then defined inductively (almost surely, all times~$t_{n,m}$ are distinct) by
     \begin{equation}\label{eq-jump-process-A}
        \begin{cases}X_n(0)=X_{n,0}, &\\
         X_n(t)=X_n(t_{n,m})+(t-t_{n,m})A_n(t)e_1,&\text{ if } t\in[t_{n,m},t_{n,m+1}),\\
         A_n(t)=A_{n,0},&\text{ if } t\in[0,t_{n,1}),\\
         A_n(t)=\bar{A}_n(t_{n,m}^-)\eta_{n,m},&\text{ if } t\in[t_{n,m},t_{n,m+1}),m\geqslant1,
       \end{cases}
     \end{equation}
     
     where~$\bar{A}_n$ is the maximizer of the function~$A\mapsto A\cdot\frac1N\sum_{l=1}^NK(X_l-X_n)A_l$. 
     \end{definition}
Since all the independent random variables~$\eta_{n,m}$ are distributed according to a law which has a density with respect to the Lebesgue measure on~$SO_3$, and the set of configurations for which this maximizer~$\bar{A}_n$ is not well defined are included in low dimensional manifolds (compared to the configuration space), we expect that this process is almost surely well defined. We do not give a detailed proof of this fact here since we are interested in derivation of kinetic models which will share the same issues, therefore we will focus on the formal derivation of these model in the case where this maximizer is well defined everywhere. The rigorous treatment of this issue is outside the scope of these lecture notes. It is even far from being well understood, even in the case of the Vicsek model, for which the only bad configurations are those with a zero average velocity. At the kinetic level, the only known global existence of solutions requires very strong assumptions of non-vanishing average velocity (which are not only assumptions on the initial conditions)~\cite{gamba2016global}.

Analogously, we can define this process in the world of quaternions.
   \begin{definition}\label{def-jump-process-q}
 We are given:
     \begin{itemize}
       \item a probability density~$M_{1}$ on~$\mathbb{H}_1$, with the property that~$M_1(q)$ only depends on~$(1\cdot q)^2=\Re(q)^2$ (we will take the von-Mises distribution defined in~\eqref{eq:von_Mises_q} in the following of the paper),
       \item some independent random variables~$S_{n,m}>0$ and~$\eta_{n,m}\in \mathbb{H}_1$, such that for~$1\leqslant n\leqslant N$ and~$m\in\mathbb{N}$,~$S_{n,m}$ is distributed according to an exponential law of parameter~$1$ and~$\eta_{n,m}$ is distributed according to~$M_{1}$,
       \item some initial positions~$X_{n,0}\in \mathbb{R}^3$ and initial body orientations~$q_{n,0}\in \mathbb{H}_1$ for~$1\leqslant n\leqslant N$.
       \end{itemize}
    Again, we define~$t_{n,m}=\sum_{0\leqslant\ell<m}S_{n,\ell}$, which corresponds to the time at which particle number~$n$ changes its orientation for the~$m$-th time. The positions and orientations are then defined inductively (almost surely, all times~$t_{n,m}$ are distinct) by
     \begin{equation}\label{eq-jump-process-q}
         \begin{cases}X_n(0)=X_{n,0},\\
         X_n(t)=X_n(t_{n,m})+(t-t_{n,m})\,q_ne_1q_n^*,&\text{ if } t\in[t_{n,m},t_{n,m+1}),\\
         q_n(t)=q_{n,0},&\text{ if } t\in[0,t_{n,1}),\\
         q_n(t)=\bar{q}_n(t_{n,m}^-)\eta_{n,m},&\text{ if } t\in[t_{n,m},t_{n,m+1}),m\geqslant1,
       \end{cases}
     \end{equation}
     where~$\bar{q}_n$ is defined in~\eqref{eq-def-Qbarn}-\eqref{eq-def-qbarn}.
          
    \end{definition}
    Once again, we expect this process to be defined almost surely, and as we remarked, thanks to Prop.~\ref{prop-dot-products}, these two definitions give rise to processes which are equivalent in law, through the map~$\Phi$. A last remark is that these processes are a particular case of Piecewise Deterministic Markov Processes (PDMP’s): between two jumps, the configuration follows an Ordinary Differential Equation (which in our case is nothing else than free transport). More comments on PDMP’s will be made in Subsection~\ref{subsec-single-jumps}.


     
 \section{Derivation of kinetic models}\label{sec:kinetic}
The aim of this section is to present a heuristic derivation of kinetic models corresponding to the limit of the particle systems when the number of particles is large. We present this derivation in the framework of rotation matrices, and we will give the corresponding kinetic models in the framework of quaternions at the end of this section.

To this aim, we introduce the so-called empirical distribution~$f^N$ of the particles as the measure 
\begin{equation*}\label{eq-empirical-distribution}
f^N(x,A,t)=\frac1{N}\sum_{i=1}^N\delta_{X_i(t)}(x)\otimes\delta_{A_i(t)}(A),
\end{equation*}
that is to say that if~$\varphi$ is a continuous and bounded function from~$\mathbb{R}^3\times SO_3(\mathbb{R})$ to~$\mathbb{R}$, the integral of~$\varphi$ with respect to this measure (at time~$t$) is given by
\begin{equation}\label{def_fN_test}\int_{\mathbb{R}^3\times SO_3(\mathbb{R})}\varphi(x,A)f^N(x,A,t)\d x \,\d A=\frac1N\sum_{i=1}^N\varphi(X_i(t),A_i(t)).\end{equation}

This function is independent of the change of numbering of particles, we say that particles are indistinguishable. 

Notice that the average orientation~$\bar{A}_n$ defined in~\eqref{eq-gooddef-Abarn} can be constructed through the empirical distribution : if we define, for a given probability density~$f$, the functions~$J_f^K$ and~$\Lambda_f^K$ by
\begin{align}
\label{def_J_K}
J^K_f(x)&=\int_{\mathbb{R}^3\times SO_3(\mathbb{R})}K(x-y)A\,f(y,A)\,\d y \,\d A,\\
\Lambda_f^K(x)&\text{ is a maximizer on }SO_3(\mathbb{R})\text{ of }A\mapsto A\cdot J^K_f,\label{def_Lambda_K}
\end{align}
we get that the definition~\eqref{eq-def-Jbarn} can be written as~$\bar{J}_n=J^K_{f^N}(X_n)$. And therefore we get~$\bar{A}_n=\Lambda_{f^N}^K(X_n)$. Therefore we obtain that the interaction between particles (which is only due to this target orientation~$\bar{A}_n$) corresponds to an interaction, for each particle, with the field generated by the empirical distribution~$f^N$. The type of limit we want to understand is called mean-field limit: when the number of particles is large, correlations between finite numbers of particles tend to vanish, and a kind of law of large numbers gives that the empirical distribution is well approached by the law of one single particle. This phenomenon is linked to the notion of propagation of chaos, and we refer to~\cite{sznitman1991topics} for an introduction. This type of limit has been rigorously shown to be valid in various models of collective behavior, such as~\cite{bolley2012meanfield} in a regularized Vicsek model, and~\cite{bolley2011stochastic,carrillo2018meanfield} in cases with less regularity. In our model, it is not straightforward to apply this strategy (due to the regularity issues for the definition of~$\Lambda_f^K$), therefore we only present a heuristic derivation of the mean-field limit one would obtain if the empirical distribution~$f^N$ converges to the law of one single particle when~$N$ is large. 

Let us now focus for the moment on a single particle model aligning with a given “target field”~$\Lambda(x,t)\in SO_3(\mathbb{R})$, and subject to some noise, as in the models given in the previous section. This corresponds to replacing~$\bar{A}_n(t)$ by~$\Lambda(X_n,t)$ in the models given by~\eqref{IBM_cont_A} (for the gradual alignment model) and~\eqref{eq-jump-process-A} (for the alignment model by orientation jumps). 

\subsection{Gradual alignment of a single individual in an orientation field}
\label{subsec-single-gradual}
We then consider the following stochastic differential equation, for the evolution of a particle at position~$X_t$ and body orientation~$A_t$, in an orientation field~$\Lambda(x,t)\in SO_3(\mathbb{R})$ :
\begin{equation}\label{gradual_single}
\begin{cases}
    \d X_t=A_te_1\, \d t,\\
    \d A_t= P_{T_{A_t}} \Lambda(X_t,t)\, \d t + 2\sqrt{D}\, P_{T_{A_t}} \circ \d B_t^{9},
\end{cases}
\end{equation}
where~$B_t^9$ is a matrix with independent coefficients given by~$9$ standard Brownian motions on~$\mathbb{R}$, and the~$\circ$ indicates that this has to be understood in the Stratonovich sense. Let us see how this last fact ensures that the orientation~$A(t)$ stays on~$SO_3(\mathbb{R})$. Thanks to the classical chain rule satisfied by Stratonovich SDE’s~\cite{hsu2002stochastic}, for a smooth function~$\varphi(x,A)$, we have
\begin{equation}
\label{Strato-chainrule}
\begin{split}\varphi(X_t,A_t)=&\varphi(X_0,A_0)\\
&+\int_0^t\left(\mathrm{D}_x\varphi(X_s,A_s)[A_se_1]+\mathrm{D}_A\varphi(X_s,A_s)[P_{T_{A_s}}\Lambda(X_s,s)]\right)\d s \\&+2\sqrt{D}\int_0^t\left(\mathrm{D}_A\varphi(X_s,A_s)[P_{T_{A_s}}(\cdot)]\right)\circ\d B_s^9,\end{split}
\end{equation}
where~$\mathrm{D}_x$ and~$\mathrm{D}_A$ are the differentials with respect to~$x\in\mathbb{R}^3$ and~$A\in\mathcal{M}$.
Now, if we take~$\varphi(x,A)=A^{\T}A-\I_3$, we get that~$\mathrm{D}_A\varphi(x,A)[H]=A^\T H+H^\T A$. Thanks to the formula~\eqref{eq-PTA}, we then get that the linear operator~$\mathrm{D}_A\varphi(x,A)[P_{T_{A}}(\cdot)]$ (from~$\mathcal{M}$ to~$\mathcal{M}$) is given by
\begin{align*}
\mathrm{D}_A\varphi(X,A)[P_{T_{A}}H]&=\frac12 A^\T (H-AH^\T A)+\frac12 (H^\T -A^\T HA^\T)A\\
&=\frac12(A^\T H \varphi(x,A)-\varphi(x,A)H^{\T}A).
\end{align*}

Therefore, if we define the linear operator~$L(Y,t):H\mapsto\frac12(A_t^\T H Y - Y H^\T A_t)$ and the process~$Y_t=\varphi(X_t,A_t)=A_t^\T A_t-\I_3$, Eq.~\eqref{Strato-chainrule} becomes
\begin{equation*}
Y_t=Y_0+\int_0^tL(Y_s,s)[\Lambda(X_s,s)]\d s +2\sqrt{D}\int_0^t L(Y_s,s)\circ\d B_s^9,
\end{equation*}
hence the process~$Y_t$ satisfies a linear SDE with initial condition~$0$, therefore it is~$0$ for all time, which means that~$A_t$ stays in~$SO_3(\mathbb{R})$ for all time.

Moreover, it is shown in Chapter~$3$ of~\cite{hsu2002stochastic} that for a manifold~$\mathcal{N}$ embedded in the euclidean space~$\mathbb{R}^d$, the generator of the SDE equation~$\d Z_t=\sigma P_{T_{Z_t}}\circ\d B_t^d$ (where~$P_{T_y}$ is the orthogonal projection on the tangent space~$T_y$ of~$\mathcal{N}$ at~$y$) is~$\frac{\sigma^2}2\Delta_{\mathcal{N}}$, where~$\Delta_{\mathcal{N}}$ is the Laplace-Beltrami operator on~$\mathcal{N}$ (the solution of this SDE is called Brownian motion on~$\mathcal{N}$). This means that for a smooth function~$\varphi$ on~$\mathcal{N}$,
\[\mathbb{E}[\varphi(Z_t)]=\mathbb{E}[\varphi(Z_0)]+\frac{\sigma^2}2\mathbb{E}\left[\int_0^t\Delta_{\mathcal{N}}\varphi(Z_s)\,\d s\right].\]
Using the Stratonovich chain rule, this means that
\[\mathbb{E}\left[\sigma\int_0^t\mathrm{D}_Z\varphi(Z_t)[P_{T_{Z_t}}(\cdot)]\circ\d B_t^d\right]=\frac{\sigma^2}2\mathbb{E}\left[\int_0^t\Delta_{\mathcal{N}}\varphi(Z_s)\,\d s\right].\]

In our case, if we write~$\mathcal{N}=SO_3(\mathbb{R})$, with the metric induced by the euclidean metric in~$\mathbb{R}^9$, this would mean that the expectation of the last term of~\eqref{Strato-chainrule} is~$2D\int_0^t\Delta_{\mathcal{N}}\varphi(X_s,A_s)\d s$. However, the metric we used for~$SO_3(\mathbb{R})$ is induced by the dot product~$(A,B)\mapsto\frac12\Tr(A^TB)$, which is half of what corresponds to the euclidean dot product in~$\mathbb{R}^9$. The Riemannian metric is then divided by~$2$, and the formula for the Laplace-Beltrami operator gives that it is then multiplied by~$2$ (recall the condensed form~$\Delta_g\varphi=\frac1{\sqrt{|g|}}\partial_i(\sqrt{|g|}g^{ij}\partial^i\varphi)$, where~$g^{ij}$  are the coefficients of the inverse of the metric tensor~$(g_{ij})_{i,j}$). Therefore we get~$\Delta_A\varphi=2\Delta_{\mathcal{N}}\varphi$. Finally, for an arbitrary test function~$\varphi$ with values in~$\mathbb{R}$, taking the expectation in Eq.~\eqref{Strato-chainrule}, and using the gradient formulation instead of the differential, we get

\begin{equation}
\label{expect-Strato-chainrule}
\begin{split}\mathbb{E}[\varphi(X_t&,A_t)]=\mathbb{E}[\varphi(X_0,A_0)]\\
&\begin{split}+\mathbb{E}\bigg[\int_0^t[\nabla_x\varphi(X_s,A_s)\cdot A_se_1&+\nabla_A\varphi(X_s,A_s)\cdot P_{T_{A_s}}\Lambda(X_s,s)\\&+D\Delta_{A}\varphi(X_s,A_s)]\d s\bigg].\end{split}\end{split}
\end{equation}

Finally, we denote by~$f(x,A,t)$ the law of such a particle at time~$t$, which is defined by the formula~$\mathbb{E}[\varphi(X_t,A_t)]=\int_{\mathbb{R}\times SO_3(\mathbb{R})}\varphi(x,a)f(x,A,t)\d A \d x$. Then, the fact that Eq.~\eqref{expect-Strato-chainrule} holds for any test function~$\varphi$, corresponds exactly to the fact that~$f$ is a weak solution of the following linear evolution equation:
\begin{equation} \label{eq:linear_VFP}
\partial_tf+(Ae_1)\cdot\nabla_xf=-\nabla_A\cdot(P_{T_A}\Lambda f)+D\Delta_Af.
\end{equation}

\subsection{Alignment by orientation jumps, for a single individual in a field}
\label{subsec-single-jumps}
We now turn to the model of alignment by orientation jumps. We then consider the following process: given some independent random variables~$S_{m}>0$ and~$\eta_{m}\in SO_3(\mathbb{R})$, such that for~$m\in\mathbb{N}$,~$S_{m}$ is distributed according to an exponential law of parameter~$1$ and~$\eta_{m}$ is distributed according to~$M_{\I_3}$, an initial position~$X_0\in\mathbb{R}^3$ and orientation~$A_0\in SO_3(\mathbb{R})$, we define~$t_{m}=\sum_{0\leqslant\ell<m}S_{\ell}$, and the position and orientation at time~$t$ are then defined inductively (almost surely, all times~$t_{m}$ are distinct) by
     \begin{equation}\label{eq-jump-process-A-single}
        \begin{cases}
         X_t=X_{t_m}+(t-t_{m})A_t e_1,&\text{ if } t\in[t_{m},t_{m+1}),\\
         A_t=A_{0},&\text{ if } t\in[0,t_{1}),\\
         A_t=\Lambda(X_{t_m},t_{m}^-)\eta_{m},&\text{ if } t\in[t_{m},t_{m+1}),\text{ with }m\geqslant1,
       \end{cases}
     \end{equation}
Another way to describe this process~$(X_t,A_t)$ is to say that it is a (non autonomous) Piecewise Deterministic Markov Process (PDMP) with jump rate~$1$, with flow~$\phi$ given by~$\phi((X,A),t)=(X+tAe_1,A)$ and with transition measure~$Q_t((X,A),\cdot)=\delta_X\otimes M_{\Lambda(X,t)}$. The only difference with classical description of PDMP’s (see for instance~\cite{azais2014piecewise} for a review of recent results), except from the fact that we work on a manifold rather than an open set of~$\mathbb{R}^d$, is that the transition measure depends on time.

Let us explain how to derive the evolution equation for the law of the process~$(X_t,A_t)$. We take once again a smooth test function~$\varphi(x,A)$, we fix a small time interval~$\delta t$, and we evaluate the expectation of~$\varphi(X_{t+\delta t},A_{t+\delta t})$. With probability~$1-\delta t+o(\delta t)$, there is no jump in~$(t,t+\delta t)$ and therefore~$A_{t+\delta t}=A_t$, and~$X_{t+\delta t}=X_t+\delta t A_te_1$. With probability~$\delta t+o(\delta t)$, there is exactly one jump at time~$s$ in~$(t,o(\delta t))$, and therefore~$A_{t+\delta t}=A_{s}$ which follows the distribution~$M_{\Lambda(X_{s},s)}$. Of course we have~$(t-s)=o(1)$. Finally, there are two or more jumps in~$(t,t+\delta t)$ with probability~$o(\delta t)$. We therefore get
\[\begin{split}\mathbb{E}[\varphi(X_{t+\delta t},&A_{t+\delta t})]=(1-\delta t+o(\delta t))\mathbb{E}[\varphi(X_t+\delta t A_te_1,A_t)]\\&+\delta t\, \mathbb{E}\left[\int_{SO_3(\mathbb{R})}\varphi(X_t+o(1),A')M_{\Lambda(X_t+o(1),t+o(1))}(A')\d A'\right] + o(\delta t),\end{split}\]
which gives, by smoothness of~$\varphi$, and if we assume that~$\Lambda$ and~$\Lambda\mapsto M_{\Lambda}$ are smooth enough, that
\[\begin{split}\frac1{\delta t}\bigg(\mathbb{E}[\varphi(X_{t+\delta t},&A_{t+\delta t})]-\mathbb{E}[\varphi(X_t,A_t)]\bigg)= \mathbb{E}[\nabla_x\varphi(X_t,A_t)\cdot A_te_1]-\mathbb{E}[\varphi(X_t,A_t)]\\&+\mathbb{E}\left[\int_{SO_3(\mathbb{R})}\varphi(X_t,A')M_{\Lambda(X_t,t)}(A')\d A'\right]+o(1),\end{split}\]
that is to say
\begin{equation}
\label{generator-PDMP}
\begin{split}\frac{\d}{\d t}\mathbb{E}[\varphi(X_t,A_t)]=\mathbb{E}\bigg[\nabla_x\varphi(X_t,&A_t)\cdot A_te_1-\varphi(X_t,A_t)\\&+\int_{SO_3(\mathbb{R})}\varphi(X_t,A')M_{\Lambda(X_t,t)}(A')\d A'\bigg].\end{split}\end{equation}
Finally, as in the previous subsection, we denote by~$f(x,A,t)$ the law of such a particle at time~$t$, defined by~$\mathbb{E}[\varphi(X_t,A_t)]=\int_{\mathbb{R}\times SO_3(\mathbb{R})}\varphi(x,a)f(x,A,t)\d A \d x$. Now, the fact that Eq.~\eqref{generator-PDMP} holds for any test function~$\varphi$ corresponds exactly to the fact that~$f$ is a weak solution of the following linear evolution equation:
\begin{equation} \label{eq:linear_VBGK}
\partial_tf+(Ae_1)\cdot\nabla_xf=\rho_fM_\Lambda-f,
\end{equation}
where
\begin{equation}\label{def-rhof}
\rho_f(x,t)=\int_{SO_3(\mathbb{R})}f(x,A,t)\d A.
\end{equation}

\subsection{Kinetic mean-field models of alignment}

Let us summarize the results of the two previous subsections: for the evolution of a particle in an orientation field~$\Lambda(x,t)$ according to one of the models~\eqref{gradual_single} or~\eqref{eq-jump-process-A-single}, the law~$f$ of the particle is evolving according to one of the (linear) kinetic equations~\eqref{eq:linear_VFP} or~\eqref{eq:linear_VBGK} which is of the form:
\begin{equation}
\label{eq:kinetic_equation_linear}
\partial_tf+(Ae_1)\cdot\nabla_xf=\Gamma_{\Lambda}(f),
\end{equation}
with
\begin{equation}\label{def_Gamma_Lambda}
\Gamma_{\Lambda}(f)=\begin{cases}-\nabla_A\cdot(P_{T_A}\Lambda f)+D\Delta_Af&\text{in the gradual alignment model},\\(\rho_fM_{\Lambda}-f)&\text{in the jump model}.\end{cases}
\end{equation}
We are now ready to provide a formal derivation of the equation satisfied by the law of one particle in the limit of a large number of particles. The heuristic is as follows: if we consider that the empirical distribution~$f^N$ of the particles converges to a deterministic law~$f$, either for the gradual alignment process~\eqref{IBM_cont_A} or for the model with orientation jumps~\eqref{eq-jump-process-A}, then each particle will evolve in the limit~$N\to\infty$ as a single particle in a orientation field~$\Lambda(x,t)$ corresponding to~$\Lambda^K_{f(t,\cdot)}(x)$, given by the formulas~\eqref{def_J_K}-\eqref{def_Lambda_K}. Therefore the evolution of the law of one particle, in the limit~$N\to\infty$, is governed by the evolution equation~\eqref{eq:kinetic_equation_linear} where~$\Lambda$ is replaced by~$\Lambda_f^K$. This gives the following (now non-linear and non-local) evolution equation:
\begin{equation} \label{eq:kinetic_equation}
\partial_tf+(Ae_1)\cdot\nabla_xf=\Gamma_{\Lambda^K_f}(f),
\end{equation}
where~$\Gamma_{\Lambda}(f)$ is defined above in~\eqref{def_Gamma_Lambda}, and with
\begin{align*}
J^K_f(x)&=\int_{\mathbb{R}^3\times SO_3(\mathbb{R})}K(x-y)A\,f(y,A)\,\d y \,\d A,\\
\Lambda^K_f(x)&\text{ maximizes }A\mapsto A\cdot J^{K}_f(x)\text{ on }SO_3(\mathbb{R}).
\end{align*}
This heuristic can be made rigorous if the map~$f\mapsto\Lambda_f$ is regular (which is not the case here, as there are configurations for which it is not even well defined). Let us quickly present the coupling method (see for instance~\cite{sznitman1991topics}) to understand how we can indeed use the law of large numbers for independent processes. The idea is first to construct a nonlinear process (for one single particle) which is the natural limit of the evolution of one particle in the particle system corresponding to~\eqref{IBM_cont_A} or~\eqref{eq-jump-process-A}, and for which the law is following the evolution equation~\eqref{eq:kinetic_equation}. For the gradual alignment process, given a Brownian motion~$B_t^9$ as in~\eqref{gradual_single}, it would be defined as follows
\begin{equation}\label{gradual_single_nonlinear}
\begin{cases}
    \d \overline{X}_t=\overline{A}_te_1\, \d t,\\
    \d \overline{A}_t= P_{T_{\overline{A}_t}} \Lambda_{f(t,\cdot)}(\overline{X}_t)\, \d t + 2\sqrt{D}\, P_{T_{\overline{A}_t}} \circ \d B_t^{9},\\
    f(t,\cdot)\text{ is the law of }(\overline{X}_t,\overline{A}_t).
\end{cases}
\end{equation}
For the orientation by jumps, given random variables~$t_m$ and~$\eta_m$ as in~\eqref{eq-jump-process-A-single}, it would be given by
\begin{equation}\label{eq-jump-process-A-single-nonlinear}
    \begin{cases}
    \overline{X}_t=\overline{X}_{t_m}+(t-t_{m})\overline{A}_t e_1,&\text{ if  }t\in[t_{m},t_{m+1}),\\
    \overline{A}_t=A_{0},&\text{ if } t\in[0,t_{1}),\\
    \overline{A}_t=\Lambda_{f(t_m,\cdot)}(\overline{X}_{t_m})\eta_{m},&\text{ if } t\in[t_{m},t_{m+1}),\text{ with }m\geqslant1,\\
    f(t,\cdot)\text{ is the law of }(\overline{X}_t,\overline{A}_t).
   \end{cases}
\end{equation}
These constructions can be seen as fixed point problems for the laws of the trajectories, and this is where the regularity of~$f\mapsto\Lambda_f$ can be used to prove a contraction property in an appropriate space. Once these processes are well defined, the second idea of the method of couplings is to introduce~$N$ of these processes~\eqref{gradual_single_nonlinear} or~\eqref{eq-jump-process-A-single-nonlinear}, for which the initial conditions and random variables (Brownian motions~$B_t^9,n$ or jump times~$t_{n,m}$ and rotations~$\eta_{n,m}$) are the same as for the particle systems~\eqref{IBM_cont_A} and~\eqref{eq-jump-process-A}. By construction, these auxiliary nonlinear processes~$(\overline{X}_n(t),\overline{A}_n(t))$ are independent and identically distributed according to the law~$f$, solution of the kinetic equation~\eqref{eq:kinetic_equation}. Therefore the last step of the coupling method is to perform estimates of the differences between the trajectories of the particle system and of the auxiliary process in order to let appear quantities reminiscent of~\eqref{def_fN_test}, but of the form
\begin{equation}\frac1N\sum_{i=1}^N\varphi(\overline{X}_i(t),\overline{A}_i(t)),\end{equation}
for which the law of large numbers applies.

Let us finish this subsection by presenting the kinetic equation we obtain (exactly in the same manner) when working with unit quaternions instead of rotation matrices. The formal mean-field limit of the particle system~\eqref{IBM_cont_q} or~\eqref{eq-jump-process-q} is given by the following evolution equation, for the density~$f(t,x,q)$ of finding a particle at position~$x$ with orientation given by the unit quaternion~$q$ at time~$t$:
\begin{equation}\label{eq:kinetic_equation_q}
\partial_t f+\Phi(q)(e_1)\cdot\nabla_xf=\Gamma_{\overline{q}^K_f}f,
\end{equation}
with
\begin{equation}\label{def_Gamma_q}
\Gamma_{\overline{q}}(f)=\begin{cases}-\nabla_q\cdot(P_{q^\perp}(\overline{q}\otimes\overline{q})q\,f)+\frac{D}4\Delta_qf&\text{in the gradual alignment model},\\
                        (\rho_fM_{\overline{q}}-f)&\text{in the jump model},\end{cases}
\end{equation}
where
\begin{align}
\label{def_rho_f_q}\rho_f&=\int_{\mathbb{H}_1}f(q)\d q,\\
\label{def_Q_fK}Q^K_f(x)&=\int_{\mathbb{R}^3\times\mathbb{H}_1}K(x-y)(q\otimes q-\tfrac14\I_4)f(y,q)\d y\,\d q,\\
\label{def_q_K}\overline{q}^K_f(x)&\text{ is an eigenvector of }Q^K_f(x)\text{ of maximal eigenvalue}.
\end{align}

\section{Macroscopic limit}\label{sec:macro}
In this section we derive the macroscopic dynamics for the kinetic equations~\eqref{eq:kinetic_equation} and~\eqref{eq:kinetic_equation_q}. This means that we are interested in the dynamics in the large time as well as large-space scale. For this we first introduce a scaling with respect to a small parameter~$\varepsilon$. We then determine the local equilibria of the collision operator, which depend on two macroscopic quantities, a density~$\rho$ and a local orientation~$\Lambda$. The final step is then the derivation of the evolution equations of these macroscopic functions~$\rho$ and~$\Lambda$. The first one comes from the conservation of mass (a collision invariant), and the second one needs more work, and can be derived using the concept of Generalized Collisional Invariants introduced in~\cite{degond2008continuum}. The subsection presenting this concept and how to use it to obtain the evolution equation for~$\Lambda$ is the main part of this section.

\subsection{Scaling}

 We introduce the macroscopic temporal and spatial variables~$(t',x')$ given by
\[
t'=\eps t, \quad x'=\eps x,
\]
where~$0<\eps\ll 1$ is a scale parameter. We also consider the following rescaling for the interaction kernel:
\[
K_\eps(x) = \frac{1}{\eps^3} K\left( \frac{x}{\eps}\right).
\]
This corresponds to localized interactions as~$\eps \to 0$ (see Rem.~\ref{rem:localized_interactions} below). Notice that
\[
\int_{\mathbb{R}^3} K_\eps (x) \ dx = \int_{\mathbb{R}^3} K (x) \ dx=1.
\]
Define the function~$f_\eps$ in the macroscopic variables as
\[
 f_\eps(t',x',A)= f(t,x,A).
\]
Our goal is to determine the dynamics for this function as~$\eps \to 0$. Firstly, one can check that the evolution equation for~$f_\eps$ is given by
\begin{align} \label{eq:rescaled_kinetic_eq}
    \eps(\partial_t f_\eps + (Ae_1) \cdot \nabla_x f_\eps) = \Gamma_{\Lambda^{K_\eps}_{f_\eps}}(f_\eps), & \qquad \mbox{(matrix formulation),}\\
    \eps (\partial_t f_\eps + \Phi(q) (e_1)\cdot \nabla_x f_\eps) = \Gamma_{\bar q^{K_\eps}_{f_\eps}} (f_\eps), & \qquad \mbox{(quaternion formulation),} \label{eq:rescaled_kinetic_quaternions}
\end{align}
where the primes have been skipped. We recall that the definition of the operators~$\Gamma_{\Lambda}$ and~$\Gamma_{\bar q}$ are given in~\eqref{def_Gamma_Lambda} and~\eqref{def_Gamma_q} respectively, and the definitions of the average orientations~$\Lambda^{K_\eps}_{f_\eps}$ and~$\bar q^{K_\eps}_{f_\eps}$ are given in~\eqref{def_J_K}--\eqref{def_Lambda_K} and~\eqref{def_Q_fK}--\eqref{def_q_K} respectively.

\medskip
Next, we expand the collision operators in the parameter~$\eps$.
\begin{lemma}[Expansion for localized interactions]
\label{lem:expansion_K}
 The following expansion holds:
    \[J^{K_\eps}_f = K_\eps *_x J_f = J_f + \mathcal{O}(\eps^2),\]
    where~$J_f(x)$ takes in account the dependence of~$f$ on the variable~$A$ only:
    \begin{equation}
    \label{def_J_f}
    J_f(x)=\int_{SO_3(\mathbb{R})}A\,f(x,A)\,\d A.
    \end{equation}
    Consequently, we can recast Eq.~\eqref{eq:rescaled_kinetic_eq} as
\begin{equation} \label{eq:rescaled_kinetic_eq2}
    \eps(\partial_t f_\eps + (Ae_1) \cdot \nabla_x f_\eps) = \Gamma_{\Lambda_{f_\eps}}(f_\eps)+ \mathcal{O}(\eps^2),
\end{equation}
where
\begin{equation}\label{def_Lambda_f}
\Lambda_f \text{ maximizes } A \mapsto A\cdot J_f \text{ on } SO_3(\mathbb{R}).
\end{equation}
Analogously, we have that
\[
Q^{K_\eps}_{f}= K_\eps *_x Q_f = Q_f + \mathcal{O}(\eps^2),
\]
    where
    \begin{equation}
\label{def_Q_f}Q_f(x)=\int_{\mathbb{R}^3\times\mathbb{H}_1}(q\otimes q-\tfrac14\I_4)f(x,q)\,\d q,
    \end{equation}
and  Eq.~\eqref{eq:rescaled_kinetic_quaternions} is recast as
\begin{equation}\label{eq:rescaled_kinetic_eq_quaternions}
\eps (\partial_t f_\eps + \Phi(q) (e_1)\cdot \nabla_x f_\eps) = \Gamma_{\bar q_{f_\eps}} f_\eps+ \mathcal{O}(\eps^2),
\end{equation}
where
\begin{equation}\label{def_q_f}
\bar q_{f} \text{ is an eigenvector of } Q_{f} \text{ of maximal eigenvalue.}
\end{equation}
\end{lemma}

\begin{remark}[Localized interactions]
\label{rem:localized_interactions}
Notice that in the leading order of the expansion of~$K_\eps$, we obtain a delta distribution in~$x$. This is why we say that this kind of rescaling corresponds to localized interactions in the limit~$\eps\to 0$.
\end{remark}

The proof of the expansions in Lem.~\ref{lem:expansion_K} is straightforward using the Taylor expansion and the fact that
\[
    \int_{\mathbb{R}^3}  xK(x) \ dx =0,
\]
for more details on this and the fact that~$\Lambda_f$ and~$\bar q _f$  are indeed defined as in the lemma, the reader is referred to~\cite[Lem. 4.1]{degond2017new} and~\cite[Lem. 4.2, Prop. 4.3]{degond2018quaternions}, respectively.

\medskip

Our goal is to investigate the limit of~$f_\eps$ as~$\eps \to 0$. Firstly, notice that, formally, from Eq.~\eqref{eq:rescaled_kinetic_eq2} we have that 
\begin{equation} \label{eq:limit_f_eps}
\mbox{if~$f_\eps(t,x,\cdot) \to f_0(t,x,\cdot)$ as~$\eps \to 0$ then~$f_0(t,x,\cdot) \in \ker(\Gamma_{\Lambda_{f_0(t,x,\cdot)}})$,}
\end{equation}
in the matrix formulation, or
 \begin{equation} \label{eq:limit_f_eps_q}
\mbox{if~$f_\eps(t,x,\cdot) \to f_0(t,x,\cdot)$ as~$\eps \to 0$ then~$f_0(t,x,\cdot) \in \ker(\Gamma_{\bar q_{f_0(t,x,\cdot)}})$,}
\end{equation}
in the quaternion formulation. For this reason, we study next the kernel of the operator~$\Gamma_\Lambda$ (which is an operator acting on functions of~$A\in SO_3(\mathbb{R})$ only) for a fixed~$\Lambda \in SO_3(\mathbb{R})$ (analogously~$\Gamma_{\bar q}$ for fixed~$\bar q \in \mathbb{H}_1$) in the following section.



\subsection{Study of the collision operator~$\Gamma$}

The goal of this subsection is to show that both the jump model and the gradual alignment model have the same type of equilibria. 
More precisely, we show the following proposition.
\begin{proposition}[Equilibria, matrix formulation]
\label{lem:comparison_equilibria}
Recall the definition of the operator~$\Gamma_\Lambda$ in Eq.~\eqref{def_Gamma_Lambda}, and the definition of the von Mises distribution~$M_{\Lambda}$ in Eq.~\eqref{eq:von_Mises_distribution}. Then, for any~$f\ge0$, we have
\begin{equation}\label{eq:kernel_A}
\Gamma_{\Lambda_f}(f) =0 \, \iff \, f=\rho M_\Lambda, \text{ for some }\rho \ge0, \Lambda\in SO_3(\mathbb{R}).
\end{equation}
Furthermore, any element~$f$ of the form~$f=\rho M_\Lambda$, with~$\rho \ge0$ and~$\Lambda\in SO_3(\mathbb{R})$ satisfies the consistency relations
\begin{equation}\label{eq:consistency_rel}
\rho_f=\rho, \quad J_{f} = \rho c_1 \Lambda, \quad \Lambda_f=\Lambda,
\end{equation}
where~$c_1\in(0,1)$ is an explicit constant, and~$\rho_f$,~$J_f$ and~$\Lambda_f$ are defined in Eq.~\eqref{def-rhof}, Eq.~\eqref{def_J_f}, and Eq.~\eqref{def_Lambda_f}, respectively.

As a consequence, both the gradual alignment model and the jump model have the same equilibria, and therefore, the same type of (formal) limit as~$\eps \to 0$: we can write
\begin{equation*}
f_0(t,x,A)=\rho(t,x) M_{\Lambda(t,x)}(A),
\end{equation*}
for some~$\rho(t,x)\ge0$ and~$\Lambda\in SO_3(\mathbb{R})$ satisfying furthermore  the consistency relations
\begin{equation}\label{eq:consistency_rel_limit}
\rho_{f_0}(t,x)=\rho(t,x), \quad J_{f_0}(t,x) = \rho(t,x) c_1 \Lambda(t,x), \quad \Lambda_{f_0}(t,x)=\Lambda(t,x).
\end{equation}
\end{proposition}
\begin{remark}[Variants in the jump-based model]In the jump-based model, the result holds true formally if we replace (in the collision operator and in the result) the von-Mises distribution by any distribution. Therefore, the jump-based model can reproduce a great variety of behaviors in terms of equilibria.
\end{remark}

The proof of Prop.~\ref{lem:comparison_equilibria} is done in~\cite{degond2017new} in the case of the gradual alignment model. We summarize here the main ideas.

We first prove the consistency relations. Let us take a function~$f$ of the form
\begin{equation}\label{eq:f_rho_Lambda}
f=\rho M_\Lambda,
\end{equation}
for some~$\rho\ge0$ and~$\Lambda\in SO_3(\mathbb{R})$. Now, one can check by direct computation that the following consistency relation holds for the average of~$M_\Lambda$ (see proof in~\cite[Lem. 4.4]{degond2017new}):
\begin{equation*} \label{eq:consistency_relation}
\int_{SO_3(\mathbb{R})} A\, M_\Lambda(A)\ dA = c_1 \Lambda, \mbox{ for some } c_1\in(0,1) \mbox{ explicit}.
\end{equation*}
With this, integrating expression~\eqref{eq:f_rho_Lambda} against 1 and~$A$ in~$SO_3(\mathbb{R})$ we obtain the two first equalities of Eq.~\eqref{eq:consistency_rel}. To conclude the proof of Eq.~\eqref{eq:consistency_rel}, the last equality is a consequence of the second one.

Now, in the gradual alignment model, it is proved in~\cite{degond2017new} that the operator~$\Gamma_{\Lambda_f}$ can be recast as
\begin{equation} \label{eq:Gamma_Fokker_Planck}
\Gamma_{\Lambda_f}(f) = D\nabla_A \cdot \left( M_{\Lambda_f}\nabla_A\left(\frac{f}{M_{\Lambda_f}} \right)\right).
\end{equation}

Using expression~\eqref{eq:Gamma_Fokker_Planck}, one can obtain that
\begin{equation*}
\ker(\Gamma_{\Lambda_f})=\{ \rho M_{\Lambda_f}\mbox{ for any }\rho=\rho(t,x)\}
\end{equation*}
 (see detailed proof of this statement in~\cite[Lem. 4.3]{degond2017new}), which, thanks to the consistency relations proved before, is equivalent to Eq.~\eqref{eq:kernel_A}.

\medskip

In the case of the jump model, from the definition of the operator~$\Gamma_{\Lambda_{f}}$ it is straightforward that its kernel is given by the functions~$f$ such that
\[
f = \rho_{f} M_{\Lambda_{f}},
\]
that is, since we have taken~$M_\Lambda$ to be the von-Mises distribution, and using again the consistency relations, exactly Eq.~\eqref{eq:kernel_A}.

\medskip
Therefore, for both models, we can use Eq.~\eqref{eq:limit_f_eps} to see that the limit~$f_0$ must be of the form
\begin{equation} \label{eq:limit}
f_0(t,x,A) = \rho(t,x) M_{\Lambda(t,x)}(A),
\end{equation}
for some~$\rho=\rho(t,x)$ and~$\Lambda= \Lambda(t,x)\in SO_3(\mathbb{R})$ to be determined, and which satisfy the consistency relations.

\bigskip
Analogously, we obtain the same kind of results for the formulation with quaternions. We write only the result on the limiting function.
\begin{lemma}[Equilibria, quaternion formulation]
\label{lem:comparison_equilibria_quaternions}
Recall the definition of the operator~$\Gamma_{\bar q}$ in Eq.~\eqref{def_Gamma_q} and of the von Mises distribution~$M_{\bar{q}}$ in Eq.~\eqref{eq:von_Mises_q}. Then, both the gradual alignment model and the jump model have the same equilibria, and therefore, the same type of limit as~$\eps \to 0$: we can write
\begin{equation*}
f_0(t,x,q)=\rho(t,x) M_{\bar q(t,x)}(q),
\end{equation*}
for some~$\rho(t,x)\ge0$ and~$\bar q(t,x)\in \mathbb{H}_1$ satisfying furthermore the consistency relations
\begin{equation}\label{eq:consistency_rel_q}
\rho_{f_0}(t,x)=\rho(t,x), \quad \bar q_{f_0}(t,x)=\bar q(t,x),
\end{equation}
where~$c_1\in(0,1)$ is an explicit constant, and~$\rho_f$ and~$\bar q_f$ are defined in Eq.~\eqref{def_rho_f_q} and Eq.~\eqref{def_q_f}, respectively.
\end{lemma}

The proof of this proposition is done in~\cite{degond2018quaternions} in the case of the gradual alignment model. We only recall the main ideas here. First, the consistency relations rely on the consistency relation satisfied by the von-Mises distribution on~$\mathbb{H}_1$, which is (see~\cite[Prop 4.4]{degond2018quaternions}):
\begin{equation} \label{eq:consistency_quaternions}
\mbox{the leading eigenvector of }\int_{\mathbb{H}_1}(q \otimes q- \frac{1}{4}I_4) M_{\bar q} \ dq \mbox{ corresponds to }\bar q.
\end{equation}
Therefore, if we take any~$f$ of the form
\begin{equation*} \label{eq:limit_quaternions}
f = \rho M_{\bar q},
\end{equation*}
multiplying this expression by 1 and~$(q\otimes q-\frac14 \I_4)$ and integrating on~$\mathbb{H}_1$ we have that
\begin{equation*}\label{eq:eq_constraints_quaternions}
\rho_{f} = \rho, \qquad Q_{f} = \rho_{f}\int_{\mathbb{H}_1} (q\otimes q- \frac{1}{4}I_4) M_{\bar q}\ dq,
\end{equation*}
where~$Q_f$ is defined in Eq.~\eqref{def_Q_f}. As a consequence of the last equality,
\begin{equation*}
\bar q_f = \bar q.
\end{equation*}
\medskip
To compute the kernel of the collision operator, in the gradual alignment model we use that the collision operator can be recast as
\[
    \Gamma_{\bar q_f}(f) = \frac{D}{4}\nabla_q\cdot \left( M_{\bar q_f} \nabla_q \left( \frac{f}{M_{\bar q_f}}\right) \right),
\]
(proved in~\cite{degond2018quaternions}). In the jump-based model the computation of the kernel is straightforward.

We then use Eq.~\eqref{eq:limit_f_eps_q} to conclude the proposition.

\bigskip
In summary, we have seen that, formally, the limit of~$f_\eps$ will be of the form~$\rho M_\Lambda$ (or~$\bar \rho M_{\bar q}$ for the quaternion case). We are left with determining the dynamics of the functions~$\rho = \rho(t,x)$,~$\bar \rho=\bar \rho(t,x)$,~$\Lambda= \Lambda(t,x)$ and~$\bar q=\bar q(t,x)$ (macroscopic quantities). This is done in the following section.

\subsection{The equation for the density~$\rho$}

We first compute the evolution for the density~$\rho=\rho(t,x)$.
We integrate the rescaled kinetic equation~\eqref{eq:rescaled_kinetic_eq2} over~$SO_3(\mathbb{R})$ and divide by~$\eps$ to obtain
\[
\partial_t \left( \int_{SO_3(\mathbb{R})} f_\eps \ dA \right) + \nabla_x\cdot \left( \int_{SO_3(\mathbb{R})} A e_1 \ f_\eps \ dA \right) =0.
\]
Importantly, the right-hand side has vanished in the integration. This cancellation reflects the fact that the total mass is conserved, i.e., the number of particles is preserved through the interactions. Now we can take formally the limit~$\eps \to 0$ and since we know the limit of~$f_\eps$ in Eq.~\eqref{eq:limit} we have that
\[
\partial_t \left( \int_{SO_3(\mathbb{R})} \rho M_\Lambda(A) \ dA \right) + \nabla_x\cdot \left( \int_{SO_3(\mathbb{R})} (A e_1) \ \rho M_\Lambda(A) \ dA \right) =0, 
\]
which corresponds to
\begin{equation} \label{eq:for_rho}
\partial_t \rho + c_1\nabla_x\cdot \left( \rho \Lambda e_1\right)  =0,
\end{equation}
given the consistency relations~\eqref{eq:consistency_rel_limit}.
The equation for the density~$\rho$ in~\eqref{eq:for_rho} corresponds to the continuity equation: the density of particles is transported with a velocity equal to~$c_1 \Lambda e_1$.

\medskip
In the formulation with quaternions, analogous computations give the same equation for~$\bar \rho$ with~$\Phi(\bar q)(e_1)$ instead of~$\Lambda e_1$, that is,
\begin{equation*} \label{eq:for_rho_q}
\partial_t \rho + c_1\nabla_x\cdot \left( \rho \, \Phi(\bar q)(e_1)\right)  =0.
\end{equation*}
\medskip

We are left with computing the evolution for~$\Lambda = \Lambda(t,x)$ and~$\bar q= \bar q(t,x)$. This is done in the following section.

\subsection{The equation for the body orientation~$\Lambda$}
The natural path to obtain an equation for~$\Lambda= \Lambda(t,x)$ is to multiply the rescaled kinetic equation~\eqref{eq:rescaled_kinetic_eq2} by~$A$; integrate this expression in~$SO_3(\mathbb{R})$; and use the consistency relations~\eqref{eq:consistency_rel_limit} at the limit~$\eps=0$. First multiplying by~$A$ and integrating we obtain
\[
\partial_t \int_{SO_3(\mathbb{R})} \hspace{-.4cm} A f_\eps \ dA + \int_{SO_3(\mathbb{R})} \hspace{-.4cm}[A \ (Ae_1\cdot \nabla_x) f_\eps]\ dA= \frac{1}{\eps} \int_{SO_3(\mathbb{R})} \hspace{-.4cm}A \Gamma_{\Lambda_{f_\eps}}(f_\eps) \ dA + \mathcal{O}(\eps),
\]
after dividing by~$\eps$ on both sides.
Notice that the limit of the first term indeed will correspond to~$c_1 \partial_t (\rho \Lambda)$ thanks to the second consistency relation in Eq.~\eqref{eq:consistency_rel_limit}. However, it is unclear how to deal with the~$\eps^{-1}$ term on the right hand side as we do not have enough information on the asymptotics of the integral (and the same difficulty arises in the quaternion framework). In classical kinetic theory (in Mathematical Physics), this difficulty does not arise: typically every macroscopic quantity corresponds to what is called a conserved quantity or collision invariant. We say that a function~$\psi$ is a collision invariant if for all~$f$ (in a reasonable class of functions)
\[
\int_{SO_3(\mathbb{R})} \psi\Gamma_{\Lambda_f} \ dA =0.
\]
We have already seen that~$\psi=1$ is a collision invariant corresponding to the total mass being conserved. However, here the body orientation~$\Lambda$ is not a conserved quantity. The same kind of non-conservative property arises in the Vicsek model for the momentum of the particles. To sort out this problem we will relax the condition of being a collision invariant taking into account the constraint given by the second equality in Eq.~\eqref{eq:consistency_rel_limit} for the limiting function. This gives rise to the concept coined as the Generalized Collision Invariant in Ref.~\cite{degond2008continuum} and that we explain in the following.

\subsubsection{The Generalized Collision Invariant.}\label{sec:GCI}
Consider the following definition:
\begin{definition}[Generalised Collision Invariant] \label{def:GCI}A function~$\psi_{\Lambda_0}$ is a Generalized Collision Invariant (GCI) associated with~$\Lambda_0 \in SO_3(\mathbb{R})$ of the operator~$\Gamma$ if it holds that
\[
\int_{SO_3(\mathbb{R})} \Gamma_{\Lambda_0}(f)\psi_{\Lambda_0}=0, \quad \mbox{ for all~$f$ such that } P_{T_{\Lambda_0}}\left( \int_{SO_3(\mathbb{R})} A\ f\ dA \right)=0.
\]
We denote by~$\mathrm{GCI}(\Lambda_0)$ this set of Generalized Collision Invariants associated with~$\Lambda_0$.

In the quaternion formulation, we say that a function~$\psi_{q_0}$ is a Generalized Collision Invariant associated to~$q_0 \in \mathbb{H}_1$ of the operator~$\Gamma$ if it holds that
\[
\int_{\mathbb{H}_1}\Gamma_{q_0}(f)\psi_{q_0}=0, \mbox{ for all~$f$ such that } P_{q_0^\perp}\left(\int_{\mathbb{H}_1} (q\otimes q-\tfrac14 \I_4 ) f(q) \ dq \ q_0 \right) =0.
\]
We also denote by~$\mathrm{GCI}(q_0)$ this set of Generalized Collision Invariants associated with~$q_0$.
\end{definition}

\begin{remark}[On the constraints on the test functions]
In the matrix formulation, one can notice that the condition on the test functions~$f$ is equivalent to saying that~$J_f\in T^\perp_{\Lambda_0}$, which is equivalent to~$J_f= \Lambda_0 S$ for some symmetric matrix~$S$ (see Prop.~\ref{prop:tangent_space_SO}). Taking~$S = \rho c_1 \I_3$ and~$\Lambda_0=\Lambda_{f_0}$, we recover the second equality in~\eqref{eq:consistency_rel_limit}. That is, the limiting function~$f_0$ is an admissible test function in the definition of the GCI (associated with~$\Lambda_{f_0}$). Something similar happens in the case of the quaternions: the conditions on the test functions~$f$ is equivalent to asking that~$P_{q_0^\perp}(Q_f\, q_0)=0$, which will hold true if~$q_0$ is an eigenvector of~$Q_f$, which is what happens, in particular, for~$f= f_0$ and~$q_0=\bar q_{f_0}$ by the second equality in~\eqref{eq:consistency_rel_q}. Therefore, the limiting function~$f_0$ is an admissible test function in the definition of the GCI (associated with~$\bar q_{f_0}$).
\end{remark}
\begin{remark}\label{rem:massGCI}
It is straightforward to see that this notion extends the notion of collision invariant. In particular, the mass~$\psi=1$, which is a collision invariant, is also a GCI (associated with~$\Lambda_0$ for any~$\Lambda_0 \in SO_3(\mathbb{R})$ if we see~$\psi$ as a function on~$SO_3(\mathbb{R})$, and associated with~$q_0$ for any~$q_0 \in \mathbb{H}_1$ if we see~$\psi$ as a function on~$\mathbb{H}_1$). Note in particular that the definition of the GCI is non-empty.
\end{remark}

We explain next how the GCI is useful. 
Multiplying Eq.~\eqref{eq:rescaled_kinetic_eq2} by a GCI associated with~$\Lambda_{f_\eps}$ and integrating with respect to~$A$ we obtain
\begin{equation*}
 \eps\int \Big(\partial_t f_\eps + (Ae_1\cdot \nabla_x) (f_\eps) \Big)\ \psi_{\Lambda_{f_\eps}}\ dA   =0.
\end{equation*}
Notice that, indeed, the right hand side vanishes since
\[
    \int_{SO_3(\mathbb{R})} \Gamma_{\Lambda_{f_\eps}}(f_\eps) \ \psi_{\Lambda_{f_\eps}} =0,
\]
given that~$f_\eps$ satisfies the condition
\[
P_{T_{\Lambda_{f_\eps}}} \left( \int_{SO_3(\mathbb{R})} A f_\eps \ dA \right) = P_{T_{\Lambda_{f_\eps}}} J_{f_\eps} =P_{T_{\Lambda_{f_\eps}}} (S_\eps\Lambda_{f_\eps}) =0,
\]
where~$ J_{f_\eps}= S_\eps\Lambda_{f_\eps}$ is the Polar Decomposition of~$J_{f_\eps}$ and~$S_\eps$ is a symmetric matrix, therefore~$J_{f_\eps}\in T^\perp_{\Lambda_{f_\eps}}$ (see Prop.~\ref{prop:tangent_space_SO}).

\medskip
Now, dividing by~$\eps$ and then making~$\eps\to 0$, using~\eqref{eq:limit} we obtain
\begin{equation} \label{eq:limit_GCI}
  \int_{SO_3(\mathbb{R})} \Big(\partial_t(\rho M_{\Lambda_{f_0}}) +(A e_1)\cdot \nabla_x(\rho M_{\Lambda_{f_0}}) \Big) \ \psi_{\Lambda_{f_0}}\ dA =0.
\end{equation}
Consequently, if we can compute the Generalized Collision Invariants in an explicit form, then we will be able to make explicit the limit given in Eq.~\eqref{eq:limit_GCI} and we will be done. This is done in the following.

\subsubsection{Description of the GCI.}
The explicit description of the GCI is given in the following proposition. For this, we need to introduce~$h= h(r)$ the unique solution (see Ref.~\cite{degond2018quaternions}) of the following differential equation on~$(-1,1)$:
\begin{align}
 & (1-r^2)^{3/2} \exp\left( \frac{2 r^2}{d}\right) \left(\frac{-4}{d}r^2-3\right) h(r)+ \frac{d}{d r} \left[  (1-r^2)^{5/2} \exp\left(\frac{2 r^2}{d}\right) h'(r) \right]\nonumber\\
& \qquad\qquad =r\, (1-r^2)^{3/2}  \exp\left(\frac{2 r^2}{d}\right). \label{eq:ode_h}
\end{align}
The function~$h$ is \emph{odd}:~$h(-r)=-h(r)$, and it satisfies for all~$r\ge0$,~$h(r)\le0$ (by maximum principle).

\begin{proposition}[Description of the GCI] \label{prop:set_GCI}
Let~$\Lambda_0 \in SO_3(\mathbb{R})$ and~$q_0\in\mathbb{H}_1$. Then, it holds that
\begin{equation*}
\begin{array}{ll}
\mathrm{GCI}(\Lambda_0) = \mathrm{span} \left\{ 1,\, \cup_{P\in \mathcal{A}} \psi_{\Lambda_0}^P\right\}, & \quad\mbox{(matrix formulation),} \\
\mathrm{GCI}(q_0) = \mathrm{span} \left\{ 1,\, \cup_{\beta\in{q_0}^\perp} \psi_{q_0}^\beta\right\},&\quad \mbox{(quaternion formulation),}
\end{array}
\end{equation*}
where, for~$P\in \mathcal{A}$ and~$\beta \in q_0^\perp$,
\begin{equation}
\begin{array}{ll}
\psi_{\Lambda_0}^P(A) = P\cdot (\Lambda_0^T A)\, \bar k(\Lambda_0 \cdot A), &\quad
\mbox{(matrix formulation)},\\
\psi_{q_0}^\beta( q):=(\beta\cdot q) \,\bar h( q\cdot q_0), &\quad \mbox{(quaternion formulation)}
\end{array}
 \label{eq:psi_beta}
\end{equation}
with~$\bar h$ given by, for~$r\in(-1,1)$,
\begin{equation*}
\bar h(r) = \begin{cases} h(r) &\text{in the gradual alignment model},\\
     r&\text{in the jump model},\end{cases}
\end{equation*}
where~$h$ is the unique solution of the differential equation~\eqref{eq:ode_h}, and~$\bar k$ given by, for~$r\in(-1/2,3/2)$,
\begin{equation}\label{def:bark}
\bar k (s)
=\frac{\bar h(\frac12\sqrt{2s+1})}{\frac12\sqrt{2s+1}}.
\end{equation}
The function~$\bar k$ is designed so that~$\bar k(\frac12+\cos\theta)=\frac{\bar h(\cos{\frac{\theta}2})}{\cos\frac{\theta}2}$. It is negative in the gradual alignment model and a constant equal to~$\bar k=1$ in the jump model.
\end{proposition}

\begin{remark}
The relation between the functions~$\bar k$ and~$\bar h$ in Eq.~\eqref{def:bark} is related to the relation between dot products, see Prop.~\ref{prop-dot-products}.
\end{remark}

The first step to prove this proposition is a characterization of the GCI in terms of the adjoint of the collision operator.

\begin{lemma}[Characterization of the GCI] A function~$\psi_{\Lambda_0}:SO_3(\mathbb{R})\to\mathbb{R}$ (resp.,~$\psi_{q_0}:\mathbb{H}_1\to\mathbb{R}$) is a GCI associated with~$\Lambda_0 \in SO_3(\mathbb{R})$ (resp., associated with~$q_0\in\mathbb{H}_1$) if and only if there exists~$P\in \mathcal{A}$ (resp.,~$\beta\in q_0^\perp$) such that~$\psi_{\Lambda_0}$ (resp.,~$\psi_{q_0}$) is solution of
\begin{equation}
\begin{gathered}\label{eq:equation_GCI}
\Gamma^*_{\Lambda_0}\psi_{\Lambda_0}(A) = P\cdot \Lambda_0^T A,\mbox{ for all } A\in SO_3(\mathbb{R}) \mbox{ (matrix formulation),}\\
\Gamma^*_{q_0}\psi_{q_0}(q)= (\beta\cdot q)(q\cdot q_0),\mbox{ for all }q\in \mathbb{H}_1 \mbox{ (quaternion formulation)},
\end{gathered}
\end{equation}
where~$\Gamma^*_{\Lambda_0}$ denotes the adjoint in~$L^2(SO_3(\mathbb{R}))$ of the operator~$\Gamma_{\Lambda_0}$ (resp., and~$\Gamma^*_{q_0}$ denotes the adjoint in~$L^2(\mathbb{H}_1)$ of~$\Gamma_{q_0}$). 
\end{lemma}

\begin{proof}
We show the proof here for the matrix formulation. For the quaternion formulation it is done analogously.
Given~$f:SO_3(\mathbb{R})\to\mathbb{R}$ and~$\Lambda_0\in SO_3(\mathbb{R})$, we have the following equivalences (in the second equivalence we use Prop.~\ref{prop:tangent_space_SO}):
\begin{align*}
 P_{T_{\Lambda_0}}\left( \int_{SO_3(\mathbb{R})} A\ f\ dA \right)=0 &\Leftrightarrow  \int_{SO_3(\mathbb{R})} A\ f\ dA  \in T^\perp_{\Lambda_0},\\
 &\Leftrightarrow  (\Lambda_0 P) \cdot \int_{SO_3(\mathbb{R})} A\ f\ dA  = 0\ \mbox{ for all } P\in \mathcal{A},\\
 &\Leftrightarrow  \int_{SO_3(\mathbb{R})} P \cdot (\Lambda_0^T A)\ f\ dA =0\ \mbox{ for all } P\in \mathcal{A},\\
 &\Leftrightarrow  f\in G^\perp,
\end{align*}
where
\[
G= \{g \in L^2(SO_3(\mathbb{R}))\ |\ g(A) = P\cdot \Lambda_0^T A, \ \mbox{for some } P \in \mathcal{A}\}.
\]

Starting from Def.~\ref{def:GCI}, we then get, for~$\psi_{\Lambda_0}:SO_3(\mathbb{R})\to\mathbb{R}$:
\begin{align*}
 \psi_{\Lambda_0}\in\mathrm{GCI}(\Lambda_0)&\Leftrightarrow\int_{SO_3(\mathbb{R})} \Gamma_{\Lambda_0}(f) \psi_{\Lambda_0}=0\quad \mbox{ for all~$f$ such that } f\in G^\perp,\\
  &\Leftrightarrow \int_{SO_3(\mathbb{R})} f \Gamma_{\Lambda_0}^*(\psi_{\Lambda_0})=0\quad \mbox{ for all~$f$ such that } f\in G^\perp,\\
 &\Leftrightarrow  \Gamma_{\Lambda_0}^*(\psi_{\Lambda_0})\in (G^\perp)^\perp= G,
\end{align*}
where~$\Gamma_{\Lambda_0}^*$ is the adjoint of~$\Gamma_{\Lambda_0}^*$ in~$L^2(SO_3(\mathbb{R}))$.
 The last equality comes from the fact that the space~$G$ is a finite-dimensional subspace of~$L^2$. The last equivalence therefore implies that~$\psi_{\Lambda_0}$ is a GCI if and only if there exists~$P\in \mathcal{P}$ such that~$\psi_{\Lambda_0}$ is solution of~\eqref{eq:equation_GCI}.\qed
\end{proof}

One can check that, in the matrix formulation, for~$\psi: SO_3(\mathbb{R})\to \mathbb{R}$, the adjoint is given by
\begin{equation}\label{eq-gammastar}
\Gamma^*_{\Lambda_0}(\psi)=\begin{cases}D\nabla_A\cdot(M_{\Lambda_0}\nabla_A \psi) &\text{(gradual alignment model)},\\
\int_{SO_3(\mathbb{R})}\psi(A)\, M_{\Lambda_0}(A)\ dA - \psi &\text{(jump model)},\end{cases}
\end{equation}
 and in the quaternion formulation we have, for~$\bar\psi:\mathbb{H}_1 \to \mathbb{R}$:
\[
\Gamma^*_{q_0}(\bar\psi)=
\begin{cases}
D\nabla_q \cdot (M_{q_0}\nabla_q \bar \psi) &\mbox{(gradual alignment model),}\\
\int_{\mathbb{H}_1}\bar\psi(q) M_{q_0}(q) \ dq - \bar \psi & \mbox{(jump model).}
\end{cases}
\]

\medskip

The end of the proof of Prop.~\ref{prop:set_GCI} for the gradual alignment model relies on the application of Lax-Milgram theorem. It is done in references~\cite{degond2017new} (for the matrix formulation) and~\cite{degond2018quaternions} (for the quaternion formulation). We do not repeat it here.

\medskip

In the case of the jump model, for the matrix formulation it is a direct check that for any~$P'\in \mathcal{A}$, the function~$\psi^{P'}_{\Lambda_0}$ defined by
\[
\psi^{P'}_{\Lambda_0}(A) = P'\cdot \Lambda^T_0 A,
\]
satisfies Eq.~\eqref{eq:equation_GCI} with~$P=-P'$ and is, therefore, a GCI. As noticed in Rem.~\ref{rem:massGCI}, the constant function~$\psi=1$ is also a GCI. Conversely, using the explicit form~\eqref{eq-gammastar} of the adjoint operator~$\Gamma^*_{\Lambda_0}$, it is also direct to see that any solution~$\psi_{\Lambda_0}$ of Eq.~\eqref{eq:equation_GCI} for some~$P\in \mathcal{A}$ satisfies
\[\psi_{\Lambda_0}(A) = -P\cdot \Lambda_0^T A + \int_{SO_3(\mathbb{R})} \psi_{\Lambda_0}(A') \ M_{\Lambda_0}(A') \ d A' \, \in \, \text{span} \left\{ 1,\, \psi_{\Lambda_0}^{-P}\right\}.\]

Analogously, one can check that for any~$\beta'\in q_0^\perp$, the function~$\psi=\psi^{\beta'}_{q_0}$ given in Eq.~\eqref{eq:psi_beta} is indeed a GCI using Eq.~\eqref{eq:equation_GCI} with~$\beta=-\beta'$ and the consistency relation~\eqref{eq:consistency_quaternions}. Conversely, one can check that any solution~$\psi$ of Eq.~\eqref{eq:equation_GCI} for some~$\beta \in q_0^\perp$ belongs to~$\text{span} \left\{ 1,\, \psi_{q_0}^{-\beta}\right\}$.

\subsubsection{Limiting equation.}
Now that we have an explicit form for the GCI, we can go back to the limiting equation~\eqref{eq:limit_GCI} (in the matrix formulation) and substitute its value. This way we have that for all~$P\in \mathcal{A}$ it holds:
\[
 \int_{SO_3(\mathbb{R})} \Big( \partial_t(\rho M_\Lambda) + (Ae_1\cdot \nabla_x)(\rho  M_\Lambda)\Big)\ (P\cdot \Lambda^T A)\ dA  =0.
\]
This is equivalent to:
\[
 P \cdot \left[  \int_{SO_3(\mathbb{R})} \Big( \partial_t(\rho M_\Lambda) + (Ae_1\cdot \nabla_x) (\rho M_\Lambda) \Big)\ \Lambda^T A\ dA \right] =0 \quad \mbox{for all }P\in \mathcal{A},
\]
which implies thanks to Prop.~\ref{prop:spacedecompositionsymmetricandanti} that
\[
\int_{SO_3(\mathbb{R})} \Big( \partial_t(\rho M_\Lambda) + (Ae_1\cdot \nabla_x) (\rho M_\Lambda) \Big)\ \Lambda^T A\ dA  \in \mathcal{S},
\]
or, in other words, 
\begin{equation*}
  \int_{SO_3(\mathbb{R})} \Big( \partial_t( \rho M_\Lambda )+ (Ae_1\cdot \nabla_x) (\rho M_\Lambda)\Big) (\Lambda^T A- A^T\Lambda)\ dA =0.
\end{equation*}
It remains only to compute this expression. This expression is exactly the same is in Ref.~\cite[Equation (4.25)]{degond2017new} with the function~$\bar \psi_0$ appearing in this reference to be taken equal to one. Therefore, here we do not repeat again the computation for this expression and put directly the result in Th.~\ref{th:macro_limit} (Eq.~\eqref{eq:macro_Lambda}) in the following section.

\subsection{Main results}
To introduce the results on the matrix formulation we need to introduce first some notation:
For a smooth function~$\Lambda$ from~$\mathbb{R}^3$ to~$SO_3(\mathbb{R})$, and for~$x\in \mathbb{R}^3$, we define the matrix~$\mathcal{D}_x(\Lambda)$ such that
\[
(w\cdot \nabla_x)\Lambda = [\mathcal{D}_x(\Lambda) w]_\times \Lambda, \qquad\mbox{for any }w\in \mathbb{R}^3.
\]
This matrix is well defined (see~\cite[Sec. 4.5]{degond2017new}). With this, we define the following first-order operators
\[
    \delta_x(\Lambda) = \Tr (\mathcal{D}_x(\Lambda)), \qquad [\mathrm{r}_x(\Lambda)]_\times = \mathcal{D}_x(\Lambda) - \mathcal{D}_x(\Lambda)^T.
\]

\medskip

In order to present the results on the quaternion formulation, we first introduce the (right) relative differential operator on~$\mathbb{H}_1$: for a function~$q=q(t,x)$ where~$q(t,x)\in\mathbb{H}_1$ and for~$\partial\in\{\partial_t,\partial_{x_1},\partial_{x_2},\partial_{x_3}\}$, let
\begin{equation} \label{eq:def_relative_operator}
\partial_{\text{rel}} q := (\partial q) q^\ast,\, \Big( = \mbox{Im}( (\partial q) q^\ast)\Big),
\end{equation}
where~$\partial q$ belongs to the orthogonal space of~$q$, and the product has to be understood in the sense of quaternions. Notice that, effectively,~$\partial_{\text{rel}} q$ is a purely imaginary quaternion, since~$\mbox{Re}((\partial q)q^*)=q \cdot \partial q=0$ (by the fact that~$q$ is a unit quaternion), and it can be identified with a vector in~$\mathbb{R}^3$. With this, we define the (right) relative space differential operators
\begin{gather*} \label{eq:def_gradient_divergence_rel}
\nabla_{x,\textnormal{rel}} q =  (\partial_{x_i,\text{rel}}q)_{i=1,2,3} =((\partial_{x_i} q) q^\ast)_{i=1,2,3} \in (\mathbb{R}^3)^3\subset \mathbb{H}^3,\\
\label{eq:def_gradient_divergence_rel2} \nabla_{x,\textnormal{rel}} \cdot q =\sum_{i=1,2,3} (\partial_{x_i,\text{rel}} q)_i =\sum_{i=1,2,3} ((\partial_{x_i}q) q^\ast)_i\in\mathbb{R},
\end{gather*}
where~$((\partial_{x_i}q) q^\ast)_i$ indicates the~$i$-th component of~$(\partial_{x_i} q) q^\ast$.

\medskip
With these notations, we can state the main result:

\begin{theorem}[(Formal) macroscopic limit]
 \label{th:macro_limit}
The following results hold true for both the jump model and the gradual alignment model. When~$\eps\to 0$ in the kinetic equations~\eqref{eq:rescaled_kinetic_eq2} (matrix representation) and~\eqref{eq:rescaled_kinetic_eq_quaternions} (quaternion representation) it holds (formally) that
\begin{align*}
&f_\eps \to f= f(t,x,A) = \rho M_\Lambda(A),\ \mbox{ with } \Lambda= \Lambda(t,x) \in SO_3(\mathbb{R}), \ \rho=\rho(t,x) \geq 0,\\
&f_\eps \to f= f(t,x, q)=\bar \rho M_{\bar q}( q),\ \mbox{ with } \bar q=\bar q(t,x)\in \mathbb{H}_1, \, \bar \rho=\bar \rho(t,x) \geq 0,
\end{align*}
for the matrix representation and the quaternion representation, respectively.
Moreover, if the convergence is strong enough and the pair functions~$(\rho,\Lambda)$,~$(\bar\rho,\bar q)$ are regular enough, then they satisfy the following systems, respectively:
\begin{align}
\partial_t \rho &+ \nabla_x \cdot \left( c_1 \rho \Lambda e_1 \right)=0, \label{eq:macro_rho}\\
\begin{split}\rho ( \partial_t \Lambda &+ c_2 ( (\Lambda e_1) \cdot \nabla_x )\Lambda) \\&+ \left[ (\Lambda e_1) \times (2c_3 \nabla_x \rho + c_4 \rho \mathrm{r}_x (\Lambda)) + c_4 \rho \delta_x (\Lambda) \Lambda e_1\right]_\times \Lambda =0,\label{eq:macro_Lambda}\end{split}
\end{align}
and
\begin{align}
\label{eq:macro_quaternion_1}\partial_t \bar \rho &+\nabla_x \cdot (c_1  e_1 ( \bar q) \bar \rho) = 0,\\
\begin{split}\bar \rho ( \partial_t \bar  q &+  c_2' (e_1(\bar q) \cdot \nabla_x) \bar q) \\&+ c_3\left[e_1(\bar q) \times\nabla_x  \bar \rho \right] \bar q \label{eq:macro_quaternion_2} +c_4\bar \rho  \left[  \nabla_{x,\textnormal{rel}}\bar  q \,e_1(\bar q) +  (\nabla_{x,\textnormal{rel}}\cdot \bar q) e_1(\bar q) \right]\bar  q =0,\end{split}
\end{align}
where the (right) relative differential operator~$\nabla_{x,\textnormal{rel}}$ is defined in Eq.~\eqref{eq:def_relative_operator}; and
\begin{equation*}
e_1(\bar q) = \Im(\bar q e_1 \bar q^*),
\end{equation*}
and where~$c_i$,~$i=1,\hdots, 4$ are explicit constants. To define them we use the following notation: for two real functions~$g$,~$w$ consider
\begin{equation*}
\langle g \rangle_{w} := \int^\pi_0 g(\theta) \frac{w(\theta)}{\int^\pi_0 w(\theta') d\theta'}\, d\theta \label{eq:definition_angles}.
\end{equation*}
Then the constants are given by
\begin{align*}
\label{eq:c1} 
c_1&= \frac{2}{3}\langle 1/2+\cos\theta\rangle_{ m(\theta)\, \sin^2(\theta/2)},\\
c_2&= \frac{1}{5} \langle 2 + 3 \cos\theta \rangle_{ m(\theta)\, \sin^4(\theta/2) \bar h (\cos(\theta/2)) \cos(\theta/2)},\\
c'_2&= \frac{1}{5} \langle 1 + 4 \cos\theta \rangle_{ m(\theta)\, \sin^4(\theta/2) \bar h (\cos(\theta/2)) \cos(\theta/2)},\\
c_3&=\frac{D}{2},\\
c_4&= \frac{1}{5}\langle 1-\cos\theta \rangle_{ m(\theta)\,\sin^4(\theta/2) \bar h (\cos(\theta/2)) \cos(\theta/2)}, 
\end{align*}
 where
\begin{equation*}
m(\theta) := \exp\left(\frac{1}{D}\left(\frac{1}{2}+\cos\theta\right) \right), \label{eq:tilde_m}
\end{equation*}
with~$\bar h$ given by, for~$r\in(-1,1)$,
\begin{equation*}
\bar h(r) = \begin{cases} h(r) &\text{in the gradual alignment model},\\
     r&\text{in the jump model},\end{cases}
\end{equation*}
where~$h$ is the unique solution of the differential equation~\eqref{eq:ode_h}.
\end{theorem}

Note that the matrix product in the fourth term of Eq.~\eqref{eq:macro_quaternion_2} has to be understood as a matrix product, giving rise to a scalar product in~$\mathbb{H}$:
\begin{equation*}\label{def:Drel_times_e1}
\nabla_{x,\textnormal{rel}}\bar q \, e_1(\bar q) = ((\partial_{x_i,\textnormal{rel}}\bar q)\cdot e_1(\bar q))_{i=1,2,3} \,.
\end{equation*}

We now state the equivalence of the matrix formulation and the quaternion formulation:
\begin{theorem}[Equivalences of the equations~\cite{degond2018quaternions}]
\label{th:equivalence_macro_equations}
Let~$\rho_0=\rho_0(x)\geq 0$. Let~$\bar q_0=\bar q_0(x)\in\mathbb{H}_1$ and~$\Lambda_0=\Lambda_0(x)\in SO_3(\mathbb{R})$ represent the same rotation, i.e.,~$\Lambda_0(x)=\Phi(\bar q_0(x))$ for all~$x\in\mathbb{R}^3$. Then the system~\eqref{eq:macro_rho}--\eqref{eq:macro_Lambda} and the system~\eqref{eq:macro_quaternion_1}--\eqref{eq:macro_quaternion_2}  are equivalent (in the sense that any solution~$( \rho,\Lambda=\Phi(\bar q))$ of the system~\eqref{eq:macro_rho}--\eqref{eq:macro_Lambda} is a solution~$(\bar \rho,\bar q)$ of~\eqref{eq:macro_quaternion_1}--\eqref{eq:macro_quaternion_2}).
\end{theorem}

Therefore the equations in the matrix formulation and in the quaternion formulation are equivalent.
For an explicit term-by-term equivalence, the reader is referred to~\cite[Sec. 5.3.3]{degond2018quaternions}.
 Moreover, we have the following corollary:
\begin{corollary}
The jump model and the gradual alignment model give rise to the same macroscopic equations with different constants when the equilibria in the jump model is given by a von-Mises distribution.
\end{corollary}

\bigskip
We conclude this section by giving a short interpretation of the macroscopic equations obtained in Th.~\ref{th:macro_limit}. For a full description and justification we refer the reader to~\cite{degond2017new,degond2018quaternions}. Since by Th.~\ref{th:equivalence_macro_equations} we know that the systems~\eqref{eq:macro_rho}--\eqref{eq:macro_Lambda} and~\eqref{eq:macro_quaternion_1}--\eqref{eq:macro_quaternion_2} are equivalent, we will restrict ourselves to interpreting the matrix formulation (for more details on the quaternion formulation the reader is referred to~\cite{degond2018quaternions}). 

Eq.~\eqref{eq:macro_rho} is the continuity equation for~$\rho$ and ensures mass conservation. The convection velocity is given by~$c_1\Lambda e_1$ and~$\Lambda e_1$ gives the direction of motion.
Eq.~\eqref{eq:macro_Lambda} gives the evolution of the mean orientation~$\Lambda$.
We remark that every term in  Eq.~\eqref{eq:macro_Lambda} belongs to the tangent space at~$\Lambda$ in~$SO(3)$; this is true for the first term since~$(\partial_t + c_2 (\Lambda e_1) \cdot \nabla_x)$ is a differential operator and it also holds for the second term because it is the product of an antisymmetric matrix with~$\Lambda$ (see Prop.~\ref{prop:tangent_space_SO}).

The term corresponding to~$c_3$ in~\eqref{eq:macro_Lambda} gives the influence of~$\nabla_x \rho$ (pressure gradient) on the body attitude~$\Lambda$. It has the effect of rotating the body around the vector directed by~$(\Lambda e_1) \times  \nabla_x \rho$ at an angular speed given by~$\frac{c_3}\rho\|(\Lambda e_1) \times  \nabla_x \rho\|$, so as to align~$\Lambda e_1$ with~$-\nabla_x\rho$ (for more details on this, see~\cite{degond2017new}). Therefore, the~$\nabla_x\rho$ term has the same effect as a pressure gradient in classical hydrodynamics. In this case  the pressure gradient has  the effect of rotating the whole body frame.

If we had that~$c_3=c_4=0$, then we would recover the Self-Organized Hydrodynamic (SOH) model. The SOH model corresponds to the macroscopic equations of the Vicsek model~\cite{degond2008continuum}.  The SOH model bears analogies with the compressible Euler equations, where~\eqref{eq:macro_rho} is obviously the mass conservation equation and~\eqref{eq:macro_Lambda} is akin to the momentum conservation equation. There are however major differences. The first one is that we preserve the constraint~$\Lambda(t)\in SO_3(\mathbb{R})$ for all times and so the mass convection speed is~$|c_1\Lambda(t) e_1 |=c_1$ for all times, while the velocity in the Euler equations is an arbitrary vector. The second one is that the convection speed~$c_2$ is a priori different from the mass convection speed~$c_1$. This difference is a signature of the lack of Galilean invariance of the system, which is a common feature of all dry active matter models.

The major novelty of the present model are the terms with constants~$c_3$ and~$c_4$. They influence the transport of the direction of motion~$\Lambda e_1$. The overall dynamics tends to align the velocity orientation~$\Lambda e_1$, not opposite to the density gradient~$\nabla_x\rho$ but opposite to a composite vector~$(c_3\nabla_x\rho+c_4\rho\,\mathrm{r}_x)$. The vector~$\mathrm{r}_x$ gives rise to an effective pressure force which adds up to the usual pressure gradient. In addition to this effective force, spatial inhomogeneities of the body attitude also have the effect of inducing a proper rotation of the frame about the direction of motion. This proper rotation is  proportional to~$\delta_x$. For an interpretation of~$\mathrm{r}_x, \ \delta_x$, see~\cite{degond2017new}.

Finally, we add the following interpretation based on the quaternion formulation. First, note that considering~$\partial=\partial_t$ the time derivative, for a function~$q=q(t,x)$ with values in~$\mathbb{H}_1$ the vector~$\partial_{t,\textnormal{rel}} q =\partial_t q \,q^{-1}$ is half of the angular velocity of a solid of orientation represented by~$q$. By analogy, the vector~$\partial_{x_i, \textnormal{rel}} q=\partial_{x_i}q \,q^{-1}$ for~$i=1,\,2,\,3$ is half of the angular variation in space of a solid of orientation represented by~$q$. Now, in the quaternion formulation the evolution equation for the body attitude can be rewritten as
\begin{equation*}
\begin{split}\bar \rho ( \partial_{t,\textnormal{rel}} \bar  q +  c_2& (e_1(\bar q) \cdot \nabla_{x,\textnormal{rel}})\bar  q) \\
&+ c_3 e_1(\bar q) \times\nabla_x  \bar \rho  +c_4\bar \rho  \left[  \nabla_{x,\textnormal{rel}}\bar  q \,e_1(\bar q) +  (\nabla_{x,\textnormal{rel}}\cdot \bar q) e_1(\bar q) \right] =0,
\end{split}
\end{equation*}
simply by multiplying Eq.~\eqref{eq:macro_quaternion_2} by~$\bar q^{-1}$ on the right. This equation lives in~$\mathbb{R}^3$ (since~$\partial_{\text{rel}} \bar q$ lives in~$\mathbb{R}^3$), and it only involves the following physical quantities: the macroscopic density~$\rho$ (and its space gradient), the macroscopic direction of movement~$e_1(\bar q)$, and the macroscopic angular time/space variations of the body attitude~$2 \partial_{\text{rel}} \bar q$. 

 \section{Conclusion}\label{sec:conclusion}
 
 In these notes, we have formally derived macroscopic models, starting from the description of particle systems, and using an intermediate kinetic model to link the two scales. The two limits ($N\to\infty$ for the particle system, and~$\varepsilon\to0$ for the rescaled kinetic equations) are formal derivations, but some steps towards a rigorous limit can be done. A way to recover a rigorous mean-field limit is to change the model in such a way that the singular behavior of the alignment is removed, as in~\cite{bolley2012meanfield}, but it introduces a phenomenon of phase transition as in~\cite{degond2013macroscopic,degond2015phase}. The study of this phase transition is an ongoing work. Another issue to have a better understanding of the limit~$\varepsilon\to0$ is to have well-posedness of the macroscopic system~\eqref{eq:macro_rho}-\eqref{eq:macro_Lambda}, so we need to study its hyperbolicity. This is also an ongoing work.
 
 \section{Acknowledgments}
P.D. acknowledges support from the Royal Society and the Wolfson foundation through a Royal Society Wolfson Research Merit Award ref WM130048; the British “Engineering and Physical Research Council” under grants ref: \\ EP/M006883/1 and EP/P013651/1; the National Science Foundation under NSF Grant RNMS11-07444 (KI-Net). P.D. is on leave from CNRS, Institut de Math\'ematiques de Toulouse, France. \\
A.F. acknowledges support from the EFI project ANR-17-CE40-0030 and the Kibord project ANR-13-BS01-0004 of the French National Research Agency (ANR), as well as from the project Défi S2C3 POSBIO of the interdisciplinary mission of CNRS, and the project SMS co-funded by CNRS and the Royal Society.\\
A.T. acknowledges support from the Kibord project ANR-13-BS01-0004 of the French National Research Agency (ANR).


\begin{thebibliography}{10}

\bibitem{azais2014piecewise}
Aza\"is, R., Bardet, J.B., G\'enadot, A., Krell, N., Zitt, P.A.: Piecewise
  deterministic {M}arkov process---recent results.
\newblock In: Journ\'ees {MAS} 2012, \emph{ESAIM Proc.}, vol.~44, pp. 276--290.
  EDP Sci., Les Ulis (2014)

\bibitem{bolley2011stochastic}
Bolley, F., Cañizo, J.A., Carrillo, J.A.: Stochastic mean-field limit:
  non-{Lipschitz} forces \& swarming.
\newblock Math. Models Methods Appl. Sci. \textbf{21}(11), 2179--2210 (2011)

\bibitem{bolley2012meanfield}
Bolley, F., Cañizo, J.A., Carrillo, J.A.: Mean-field limit for the stochastic
  {Vicsek} model.
\newblock Appl. Math. Lett. \textbf{3}(25), 339--343 (2012)

\bibitem{bostan2017reduced}
Bostan, M., Carrillo, J.A.: Reduced fluid models for self-propelled particles
  interacting through alignment.
\newblock Math. Models Methods Appl. Sci. \textbf{27}(7), 1255--1299 (2017)

\bibitem{carrillo2018meanfield}
Carrillo, J.A., Choi, Y., Hauray, M., Salem, S.: Mean-field limit for
  collective behavior models with sharp sensitivity regions.
\newblock J. Europ. Math. Soc.  (2018).
\newblock To appear

\bibitem{cavagna2014flocking}
Cavagna, A., Del~Castello, L., Giardina, I., Grigera, T., Jelic, A., Melillo,
  S., Mora, T., Parisi, L., Silvestri, E., Viale, M., et~al.: Flocking and
  turning: a new model for self-organized collective motion.
\newblock J. Stat. Phys. \textbf{158}(3), 601--627 (2014)

\bibitem{cercignani2013mathematical}
Cercignani, C., Illner, R., Pulvirenti, M.: The mathematical theory of dilute
  gases, vol. 106.
\newblock Springer Science \& Business Media (2013)

\bibitem{constantin2010onsager}
Constantin, P.: The onsager equation for corpora.
\newblock Journal of Computational and Theoretical Nanoscience \textbf{7}(4),
  675--682 (2010)

\bibitem{degond2004macroscopic}
Degond, P.: Macroscopic limits of the {B}oltzmann equation: a review.
\newblock In: Modeling and Computational Methods for Kinetic Equations, pp.
  3--57. Springer (2004)

\bibitem{degond2013macroscopic}
Degond, P., Frouvelle, A., Liu, J.G.: Macroscopic limits and phase transition
  in a system of self-propelled particles.
\newblock J. Nonlinear Sci. \textbf{23}(3), 427--456 (2013)

\bibitem{degond2015phase}
Degond, P., Frouvelle, A., Liu, J.G.: Phase transitions, hysteresis, and
  hyperbolicity for self-organized alignment dynamics.
\newblock Arch. Ration. Mech. Anal. \textbf{216}(1), 63--115 (2015)

\bibitem{degond2017new}
Degond, P., Frouvelle, A., Merino-Aceituno, S.: A new flocking model through
  body attitude coordination.
\newblock Math. Models Methods Appl. Sci. \textbf{27}(06), 1005--1049 (2017)

\bibitem{degond2018quaternions}
Degond, P., Frouvelle, A., Merino-Aceituno, S., Trescases, A.: Quaternions in
  collective dynamics.
\newblock Multiscale Mod. Simul. \textbf{16}(1), 28--77 (2018)

\bibitem{degond2013hydrodynamic}
Degond, P., Liu, J.G., Motsch, S., Panferov, V.: Hydrodynamic models of
  self-organized dynamics: derivation and existence theory.
\newblock Methods Appl. Anal. \textbf{20}(2), 89--114 (2013)

\bibitem{degond2017continuum}
Degond, P., Manhart, A., Yu, H.: A continuum model for nematic alignment of
  self-propelled particles.
\newblock Discrete Contin. Dyn. Syst. Ser. B \textbf{22}(4), 1295--1327 (2017)

\bibitem{degond2008continuum}
Degond, P., Motsch, S.: Continuum limit of self-driven particles with
  orientation interaction.
\newblock Math. Models Methods Appl. Sci. \textbf{18}, 1193--1215 (2008)

\bibitem{degond2011macroscopic}
Degond, P., Motsch, S.: A macroscopic model for a system of swarming agents
  using curvature control.
\newblock J. Stat. Phys. \textbf{143}(4), 685--714 (2011)

\bibitem{degond2015multi}
Degond, P., Navoret, L.: A multi-layer model for self-propelled disks
  interacting through alignment and volume exclusion.
\newblock Math. Models Methods Appl. Sci. \textbf{25}(13), 2439--2475 (2015)

\bibitem{dimarco2016selfalignment}
Dimarco, G., Motsch, S.: Self-alignment driven by jump processes: Macroscopic
  limit and numerical investigation.
\newblock Math. Models Methods Appl. Sci. \textbf{26}(07), 1385--1410 (2016)

\bibitem{doi1999theory}
Doi, M., Edwards, S.F.: The Theory of Polymer Dynamics, \emph{International
  Series of Monographs on Physics}, vol.~73.
\newblock Oxford University Press (1999)

\bibitem{ferdinandy2017}
Ferdinandy, B., Ozog\'any, K., Vicsek, T.: Collective motion of groups of
  self-propelled particles following interacting leaders.
\newblock Phys. A \textbf{479}, 467--477 (2017)

\bibitem{frouvelle2012continuum}
Frouvelle, A.: A continuum model for alignment of self-propelled particles with
  anisotropy and density-dependent parameters.
\newblock Math. Models Methods Appl. Sci. \textbf{22}(7), 1250,011, 40 (2012)

\bibitem{gamba2016global}
Gamba, I.M., Kang, M.J.: Global weak solutions for {K}olmogorov-{V}icsek type
  equations with orientational interactions.
\newblock Arch. Ration. Mech. Anal. \textbf{222}(1), 317--342 (2016)

\bibitem{hsu2002stochastic}
Hsu, E.P.: Stochastic Analysis on Manifolds, \emph{Graduate Series in
  Mathematics}, vol.~38.
\newblock American Mathematical Society, Providence, Rhode Island (2002)

\bibitem{sarlette2009autonomous}
Sarlette, A., Sepulchre, R., Leonard, N.E.: Autonomous rigid body attitude
  synchronization.
\newblock Automatica \textbf{45}(2), 572--577 (2009)

\bibitem{sone2012kinetic}
Sone, Y.: Kinetic theory and fluid dynamics.
\newblock Springer Science \& Business Media (2012)

\bibitem{sznitman1991topics}
Sznitman, A.S.: Topics in propagation of chaos.
\newblock In: École d’Été de Probabilités de Saint-Flour XIX --- 1989,
  \emph{Lecture Notes in Mathematics}, vol. 1464, pp. 165--251. Springer,
  Berlin (1991)

\bibitem{vicsek1995novel}
Vicsek, T., Czirók, A., Ben-Jacob, E., Cohen, I., Shochet, O.: Novel type of
  phase transition in a system of self-driven particles.
\newblock Phys. Rev. Lett. \textbf{75}(6), 1226--1229 (1995)

\end{thebibliography}
\end{document}